%% file: colt2020-sample.tex
\documentclass[final]{colt2020} 

\title[Polyak steps with momentum]{Complexity Guarantees for Polyak Steps with Momentum}
\usepackage{times}
\usepackage{algorithm}
\usepackage{algorithmic}
\usepackage{nccmath}
\renewcommand{\algorithmicrequire}{\textbf{Input:}} 
\renewcommand{\algorithmicensure}{\textbf{Output:}}

\newcommand{\N}{\mathbb{N}}

\usepackage {tikz}
\usepackage{pgfkeys}
\usetikzlibrary{shapes,arrows}
\usepackage{pgfplots}
\pgfplotsset{compat=1.13}
\pgfplotsset{plotOptions/.style={%
		label style={font=\scriptsize},
		legend style={font=\scriptsize},
		tick label style={font=\scriptsize},
		solid,
		very thick
	}}
\definecolor{colorP1}{RGB}{55,126,184}  
\definecolor{colorP2}{RGB}{228,26,28}  
\definecolor{colorP3}{RGB}{152,78,163} 
\definecolor{colorP4}{RGB}{77,175,74}  
\definecolor{colorP5}{RGB}{250, 150, 10} 
\definecolor{colorP6}{RGB}{139,69,19}

\coltauthor{%
 \Name{Mathieu \textsc{Barré}} \Email{mathieu.barre@inria.fr}\\
 \addr INRIA, D\'epartement d'informatique de l'ENS, Ecole normale sup\'erieure, CNRS, PSL Research University, Paris, France
 \AND
 \Name{Adrien \textsc{Taylor}} \Email{adrien.taylor@inria.fr}\\
 \addr INRIA, D\'epartement d'informatique de l'ENS, Ecole normale sup\'erieure, CNRS, PSL Research University, Paris, France
 \AND
 \Name{Alexandre d'\textsc{Aspremont}} \Email{aspremon@ens.fr}\\
 \addr D\'epartement d'informatique de l'ENS, Ecole normale sup\'erieure, CNRS, PSL Research University, Paris, France.
}
\begin{document}
\input defs.tex

\maketitle
\begin{abstract}
In smooth strongly convex optimization, knowledge of the strong convexity parameter is critical for obtaining simple methods with accelerated rates. In this work, we study a class of methods, based on Polyak steps, where this knowledge is substituted by that of the optimal value,~$f_*$. We first show slightly improved convergence bounds than previously known for the classical case of simple gradient descent with Polyak steps, we then derive an accelerated gradient method with Polyak steps and momentum, along with convergence guarantees.
\end{abstract}


\section{Introduction}
We focus on unconstrained optimization problems of the form
\[ 
\min_{x\in\reals^n} f(x),
\]
where $f$ is strongly convex and has a Lipschitz continuous gradient with respect to the Euclidean norm. Very broadly speaking, the current numerical toolbox to solve these convex minimization problems contains two types of methods. On one hand, simple numerical schemes with explicit albeit conservative theoretical guarantees. These include gradient methods and their accelerated variants, and require knowing problem {\em parameters}, such as strong convexity parameters, or H\"olderian error bounds \citep{Bolt07}. On the other hand, {\em adaptive methods}, such as conjugate gradients or quasi-Newton, adapting much better to the objective function by estimating some of its regularity properties. For these methods, we typically have no theoretical justification for their improved performances or no computational complexity bounds at all.

Empirically, adaptive methods often perform significantly better than their parametric counterparts, and, by nature, require much less tuning. For example, roughly estimating regularity constants on-the-fly and plugging these estimates in parametric algorithms often produces fast algorithms with no theoretical guarantees. This phenomenon is illustrated in Figure~\ref{fig:intro-adap} on logistic regression.

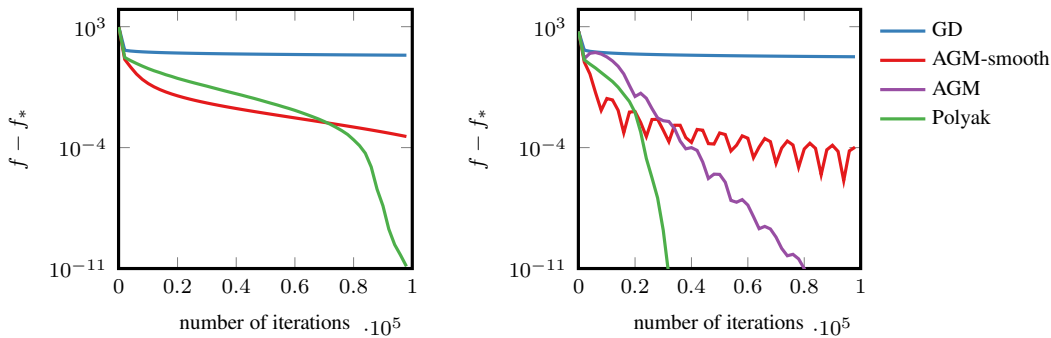
\begin{figure}[ht!]
	\centering
	    \begin{tabular}{cc}
		\begin{tikzpicture}
		\begin{semilogyaxis}[legend pos=outer north east, legend style={draw=none},legend cell align={left}, ylabel={$f-f_*$},xlabel={number of iterations}, plotOptions,  ymin=1e-11, ymax=1e4,xmin=-100,xmax=100000,width=.36\linewidth, height=.33\linewidth]
		\addplot [color=colorP1] table [x=iter,y=gd] {figures/logitSonarReg1e-7.txt};
		\addplot [color=colorP2] table [x=iter,y=apg] {figures/logitSonarReg1e-7.txt};
		\addplot [color=colorP4] table [x=iter,y=pol] {figures/logitSonarReg1e-7.txt};
		\end{semilogyaxis}
		\end{tikzpicture}
		&
		\begin{tikzpicture}
		\begin{semilogyaxis}[legend pos=outer north east, legend style={draw=none},legend cell align={left}, ylabel={$f-f_*$},xlabel={number of iterations}, plotOptions,  ymin=1e-11, ymax=1e4,xmin=-100,xmax=100000,width=.35\linewidth, height=.33\linewidth]
		\addplot [color=colorP1] table [x=iter,y=gd] {figures/logitSonarReg1e-4.txt};
		\addlegendentry{GD}
		\addplot [color=colorP2] table [x=iter,y=apg] {figures/logitSonarReg1e-4.txt};
		\addlegendentry{AGM-smooth}
		\addplot [color=colorP3] table [x=iter,y=apgl] {figures/logitSonarReg1e-4.txt};
		\addlegendentry{AGM}
		\addplot [color=colorP4] table [x=iter,y=pol] {figures/logitSonarReg1e-4.txt};
		\addlegendentry{Polyak}
		\end{semilogyaxis}
		\end{tikzpicture}
		\end{tabular}
		\vspace{-.5cm}
		\caption{Convergence of gradient descent (GD), accelerated gradient method for smooth optimization (AGM-smooth) \citep{Nest83b}, accelerated gradient method with constant momentum (AGM)---described below as Algorithm~\ref{algo:agdp} with~\eqref{eq:costmom}---where the momentum is set using the value of the regularization parameter and gradient method with Polyak steps (Polyak). Experiments on regularized logistic regression for the Sonar dataset without any tuning of the methods. Left:~regularization parameter $10^{-7}$. Right:~regularization parameter $10^{-4}$. For Polyak steps, the best iterate is displayed. Observe that Polyak method is a (non-accelerated) adaptive method, which performs comparatively well against accelerated schemes.}
		\label{fig:intro-adap}
		\vspace{-.7cm}
\end{figure}

Although many advances have been made in designing optimization schemes adaptive to some types of parameters (e.g., Lipschitz constants, see discussions below), these results still leave a huge gap between theory and practice (as in Figure~\ref{fig:intro-adap}). In particular, estimating strong convexity coefficients while preserving convergence guarantees remains a challenging issue. Restart schemes are probably the most effective option among existing approaches for adapting to this type of parameters and do provide improved complexity estimates without any knowledge of strong convexity parameters, at the expense of a log scale grid search. However, while on paper the complexity of these schemes is nearly optimal, the presence of an outer loop clearly limits their practical effectiveness and their capacity to adapt to the function's local regularity, which leaves a lot of margin for improvement, on the numerical front. Producing single loop algorithms adapting to local strong convexity (or H\"olderian error bounds) and have nearly optimal complexity bounds is an important open problem which is the main focus of this work.

Here, we study the complexity of adaptive methods using Polyak steps, estimating the momentum term using information on the optimum objective value $f_*$ instead of the strong convexity constant. In some scenarios, such as ``interpolation'' in machine learning problems, the value of $f_*$ is known a priori (usually zero), and estimating it is much easier than estimating strong convexity, see e.g.,~\citep{Asi19b} for a recent discussion on these model assumptions.

The obvious next research question in this direction is to substitute knowledge on $f_*$ by weaker bounds. A first step in this direction is for example \citep{Haza19b} which uses successive refinements of a lower bound on $f_*$. As it is, the proof in \citep{Haza19b} contains several errors, but can be fixed. We hope, and believe, that such a mechanism could be used together with our momentum version of the Polyak steps.

\input{related_work.tex}

\subsection{Contributions} 
We develop and analyze an accelerated variant of the gradient method with Polyak steps that includes a momentum term and has better dependence on the condition number. We believe the Performance Estimation Program (PEP) technique used for obtaining the worst-case convergence guarantees is also of independent interest. As a byproduct, we also slightly improve convergence bounds for variants of the classical gradient method with Polyak steps (i.e. without momentum).  


\section{Classical Polyak Steps and Variants}
We denote $f_*$ the minimum of $f$. Let $0\leq \mu < L$, the class of $L$-smooth and $\mu$-strongly convex functions is denoted $\mathcal{F}_{\mu,L}$. Functions in this class satisfy (see e.g.,~\citep{Nest18}) $\forall x,y\in\reals^n$:
\begin{equation*}
    \begin{aligned}
     f(y) &\leq f(x) + \langle\nabla f(x),y-x \rangle + \tfrac{L}{2}\|y-x\|^2 \quad\text{(smoothness)},\\
     f(y) &\geq f(x) + \langle\nabla f(x),y-x \rangle + \tfrac{\mu}{2}\|y-x\|^2 \quad\text{(strong convexity)}.
    \end{aligned}
    \end{equation*}

Let us start with complexity bounds for gradient methods with Polyak steps for smooth and strongly convex optimization problems. Note that Polyak step sizes are usually discussed in the nondifferentiable setting---see \citep[Section 5.3.2]{polyak1987introduction} or \citep{nedic2001incremental,Boyd03b}. We first recall the complexity of the gradient method with Polyak steps in the smooth strongly convex case, then derive similar bounds for two variants. For the first variant, we scale the steps by a factor two compared to standard Polyak steps, yielding a simple convergence proof with slightly improved theoretical guarantees. The second variant is a descent method, where the complexity bound is written in terms of the primal gap. We delay a full discussion of the proof mechanisms to Section~\ref{sec:proofsmech}, and the proofs themselves to the appendix.
\begin{algorithm}[!ht]
\caption{Adaptive gradient method}
\label{algo:generic_Polyak}
\begin{algorithmic}
\STATE \algorithmicrequire\;$x_0 \in \reals^n$, $f_*\in \reals$
\FOR{$k \geq 0$}
\STATE compute $\gamma_k$
\STATE $x_{k+1} = x_k - \gamma_k\nabla f(x_k)$
\ENDFOR
\STATE \algorithmicensure\; $x_{k+1}$
\end{algorithmic}
\end{algorithm}\vspace{-.5cm}
\begin{align}
\text{Regular Polyak steps:}\quad &&&\gamma_k=\tfrac{f(x_k)-f_*}{\|\nabla f(x_k)\|^2} \label{eq:Polyakstep1}\tag{Polyak}\\
\text{Polyak steps, variant I:}\quad &&&\gamma_k=2\tfrac{f(x_k)-f_*}{\|\nabla f(x_k)\|^2}\label{eq:Polyakstep2}\tag{Variant I}\\
\text{Polyak steps, variant II:}\quad &&&\gamma_k = \left(2-\tfrac{\|\nabla f(x_k)\|^2}{2L(f(x_k)-f_*)}\right)/L\label{eq:Polyakstep3}\tag{Variant II}
\end{align}

The classical step size rule~\eqref{eq:Polyakstep1} was mostly studied in the nonsmooth convex case~\citep{polyak1987introduction}. For smooth strongly convex problems, it is known (see e.g.,~\citep{Haza19b}) that
\begin{equation}
f(x_N)-f_*\leq (1-\tfrac\mu{L})^N\tfrac{L\|x_0-x_*\|^2}{2} \label{eq:Polyak1Guarantee}.
\end{equation}

The two following propositions show that different step sizes policies (namely \eqref{eq:Polyakstep2} and \eqref{eq:Polyakstep3}) produce slightly improved convergence rates, matching the best known rates for gradient methods with known $\mu$ and $L$. The $\gamma_k$ are always well defined except when $x_k$ has a zero gradient, in this case we can simply stop the method since we have reached optimality. When it is well defined, $\gamma_k \in [\tfrac{1}{L},\tfrac{1}{\mu}]$ for \eqref{eq:Polyakstep2} and $\gamma_k \in [\tfrac{1}{L},\tfrac{2-\mu/L}{L}]$ for \eqref{eq:Polyakstep3}. First, let us state that if we seek to decrease the distance to the optimal point,~\eqref{eq:Polyakstep2} provides a rate that matches that of gradient descent with optimal (non-adaptive) step sizes~\citep{Nest18}. 

\begin{proposition}[Appendix~\ref{proof:polyak1}]\label{prop:polyak1}
Let $f \in \mathcal{F}_{\mu,L}$ and consider Algorithm~\ref{algo:generic_Polyak} with step-sizes~\eqref{eq:Polyakstep2}. Then, for any $x_0 \in \reals^n$ and $N \in \N$, such that the sequence $\{\gamma_k\}_k$ is well defined, it holds that 
\[\|x_N -x_*\|^2 \leq \left(\prod_{k=0}^{N-1}\rho(\gamma_k)\right)\|x_0-x_*\|^2,\]
where $\rho(\gamma) = \tfrac{(\gamma L -1)(1-\gamma\mu)}{\gamma(L+\mu)-1}$, and $\underset{\gamma \in [\tfrac{1}{L},\tfrac{1}{\mu}]}{\max}\; \rho(\gamma) = \tfrac{(L-\mu)^2}{(L+\mu)^2}$. Otherwise $\nabla f(x_k) = 0$ with $k\in[0,N]$.
\end{proposition}

If on the other hand we seek to decrease the primal gap,~\eqref{eq:Polyakstep3} provides a rate that matches that of gradient descent with exact line search~\citep{de2017worst}, at the expense of knowledge on L.
\begin{proposition}[Appendix~\ref{proof:polyak2}]\label{prop:polyak2}
Let $f \in \mathcal{F}_{\mu,L}$ and consider Algorithm~\ref{algo:generic_Polyak} with step-sizes~\eqref{eq:Polyakstep3}. Then, for any $x_0 \in \reals^n$ and $N \in \N$, such that the sequence $\{\gamma_k\}_k$ is well defined, it holds that 
\[f(x_N) -f_* \leq \left(\prod_{k=0}^{N-1}\rho(\gamma_k)\right)(f(x_0)-f_*),\]
where $\rho(\gamma) = (L\gamma - 1)\left(L\gamma(3-\gamma(L+\mu))-1\right)$, and $\underset{\gamma \in [\tfrac{1}{L},\tfrac{2L-\mu}{L^2}]}{\max}\; \rho(\gamma) = \tfrac{(L-\mu)^2}{(L+\mu)^2}$.\\Otherwise $\nabla f(x_k) = 0$ with $k \in [0,N]$.
\end{proposition}

In the following section, we study variants of those methods, where we aim to speed up convergence by incorporating a momentum term. Those methods follow in spirit the line of works on Nesterov's acceleration~\citep{Nest13b}, where we supersede knowledge of $\mu$ by that of $f_*$.
\section{Acceleration with Polyak momentum}\label{sec:acc}
In the following, AGM refers to the Accelerated Gradient Method with momentum introduced by Nesterov~\citep{Nest83b,Nest18}.
We are interested in optimizing a function $f \in \mathcal{F}_{\mu,L}$ without any information on the strong convexity constant $\mu$. However, as in the Polyak gradient method, we rely on the knowledge of $f^*$. We describe a single loop adaptive accelerated method (i.e. without restarts), with convergence rate of order $1-\left({\mu}/{L}\right)^{3/4}$, compared with $1-\mu/{L}$ for gradient descent, and $1-\left({\mu}/{L}\right)^{1/2}$ for its accelerated version with perfect knowledge of $\mu$.
\begin{algorithm}[!ht]
\caption{Accelerated gradient method (AGM)}
\label{algo:agdp}
\begin{algorithmic}
\STATE \algorithmicrequire\;$x_0 \in \reals^n$, $f_*\in \reals$, $L$ smoothness constant.
\STATE $y_0 = x_0$, 
\FOR{$k \geq 0$}
\STATE $y_{k+1} = x_k - \frac{1}{L}\nabla f(x_k)$
\STATE  compute $\tilde{\mu}_k$ and $\beta_k=\tfrac{\sqrt{L}-\sqrt{\tilde{\mu}_k}}{\sqrt{L}+\sqrt{\tilde{\mu}_k}}$
\STATE $x_{k+1} = y_{k+1} + \beta_k(y_{k+1}-y_k)$
\ENDFOR
\STATE \algorithmicensure\; $y_{k+1}$
\end{algorithmic}
\end{algorithm}
\begin{align}
\text{Constant momentum:}\hspace{0.1cm} &&&\tilde{\mu}_k=\mu \label{eq:costmom}\tag{Const-mom}&\\
\text{Polyak Acc., variant I:}\hspace{0.1cm} &&&\tilde{\mu}_k=\tfrac{\|\nabla f(y_{k+1})\|^2}{2(f(y_{k+1})-f_*)} ,\; \label{eq:Acc1}\tag{Acc. Variant I}&\\
\text{Polyak Acc., variant II:}\hspace{0.1cm} &&&\tilde{\mu}_k=\left\{\begin{array}{ll}
    +\infty \quad& \text{if } k=-1 \\
    \min\left(\tilde{\mu}_{k-1},\tfrac{\|\nabla f(y_{k+1})\|^2}{2(f(y_{k+1})-f_*)}\right) & \text{otherwise} 
\end{array}\right.\; \label{eq:Acc2}\tag{Acc. Variant II}&
\end{align}

Algorithm~\ref{algo:agdp} is based on the AGM algorithm \citep{Nest18}, in which the knowledge of $\mu$ is essential to set the constant momentum term $\beta_k=\beta_*=(\sqrt{L}-\sqrt{{\mu}})/(\sqrt{L}+\sqrt{{\mu}})$. Common convergence guarantees require  a lower bound on the strong convexity. As a first step towards producing adaptive versions of AGM, Lemma~\ref{lem:notbroken} and Corollary~\ref{cor:notbroken} below guarantee that AGM with any momentum factor $\beta_k$ in $[0,1]$ converges at least as fast as the classical gradient method.

\begin{lemma}[Convergence of AGM with bad momentum, Appendix~\ref{proof:notbroken}]\label{lem:notbroken}
Let $f \in \mathcal{F}_{\mu,L}$, some iteration number $k\in\N$, and consider Algorithm~\ref{algo:agdp} with $\beta_k \in [0,1]$. Then, for any $x_k,y_k \in \reals^n$, it holds that
\BEQ
V(x_{k+1},y_{k+1}) \leq \rho V(x_k,y_k)
\EEQ
where $V(x,y) = \frac{L-\mu}{2}\|x - y\|^2 + f(y)-f_* $ and $\rho = 1-\frac{\mu}{L}$.
\end{lemma}
We then get the following corollary on the primal gap.
\begin{corollary}\label{cor:notbroken}
Let $f \in \mathcal{F}_{\mu,L}$, a number of iterations $N\in \N$, and consider Algorithm~\ref{algo:agdp} with a sequence $\{\beta_k\}_k$ satisfying $\beta_k\in\, [0,1]$ for all $k\in[1,N]$. Then, for any $x_0 \in \reals^n$, it holds that
\[ f(y_N) -f_* \leq \left(1 -\frac{\mu}{L} \right)^N(f(x_0)-f_*).\]
\end{corollary}
\begin{proof}
Direct from Lemma~\ref{lem:notbroken} with $x_0=y_0$.
\end{proof}

This result shows the robustness of AGM with respect to the momentum parameter. 
Adaptive strategies, that modify the momentum term in the algorithm automatically, thus at least enjoy the gradient method's convergence rate when $\beta_k$ is kept within the interval $[0,1]$---this is the case for both~\eqref{eq:Acc1} and~\eqref{eq:Acc2}. To our knowledge, only non-blowup properties \citep[Lemma 1]{Lin14} were known when overestimating $\mu$.

The momentum term in \eqref{eq:Acc1} was designed using the inverse of Polyak's step as an estimate of the strong convexity parameter. The motivation for this choice of strong convexity estimate is the fact that under some mild assumptions on $f$ (i.e., for quadratic or self-concordant $f$), the quantity $\frac{\|\nabla f(z_{k})\|^2}{2(f(z_{k})-f_*)}$ converges to the strong convexity constant at optimum when the $z_k$ are iterates of gradient descent algorithm with step-size 1/L.

 In order for $\tilde{\mu}_k$ to be always defined and within the interval $[\mu,L]$, we assume that iterates never reach exactly optimality. Under this condition we have $\beta_k \in [0,\beta_*]$ and Corollary~\ref{cor:notbroken} readily applies to both~\eqref{eq:Acc1} or~\eqref{eq:Acc2}. However, this result can be improved for those particular choices, as described in Lemma~\ref{lem:acc_classic} and Proposition~\ref{prop:acd_as_gd}, as the rate can be expressed in terms of the local $\tilde{\mu}_k$ instead of $\mu$.

\begin{lemma}[Appendix~\ref{proof:acc_classic}]\label{lem:acc_classic} Let $f \in \mathcal{F}_{0,L}$, some iteration number $k\in\N$, and consider Algorithm~\ref{algo:agdp} with either~\eqref{eq:Acc1} or~\eqref{eq:Acc2}. For any $x_k,y_k \in \reals^n$ such that $\tilde{\mu}_k$ well defined, it holds that
\BEQ
V(x_{k+1},y_{k+1}) \leq \rho(\tilde{\mu}_k) V(x_k,y_k)
\EEQ
where $V(x,y) = \frac{L}{2}\|x - y\|^2 + f(y)-f_* $ and $\rho(\tilde{\mu}) = \frac{1}{1+\frac{\tilde{\mu}}{L}}$. Otherwise $\nabla f(y_{k+1}) = 0$.
\end{lemma}
\begin{proposition}\label{prop:acd_as_gd} Let $f \in \mathcal{F}_{0,L}$, some number of iterations $N\in\N$, and consider Algorithm~\ref{algo:agdp} with either~\eqref{eq:Acc1} or~\eqref{eq:Acc2}. Then, for any $x_0 \in \reals^n$, such that the sequence $\{\tilde{\mu}_k\}_k$ is well defined , it holds that 
\[f(y_N) -f_* \leq \left(\prod_{k=0}^{N-1}\rho(\tilde{\mu}_k)\right)(f(x_0)-f_*)\]
where $\rho(\tilde{\mu}) = \frac{1}{1+\frac{\tilde{\mu}}{L}}$. Otherwise $\nabla f(y_k) = 0$ with $k \in [0,N]$.
\end{proposition}
\begin{proof}
Use Lemma~\ref{lem:acc_classic} recursively and notice that $V(x_0,y_0) = f(x_0)-f_*$.
\end{proof}

In fact, these results on~\eqref{eq:Acc1} and~\eqref{eq:Acc2} also hold under H\"olderian error bounds~\citep{Bolt07,bolte2017error} (also known as Kurdyka-{\L}ojasewicz, Polyak-{\L}ojasewicz, quadratic growth, etc.)  which require the existence of $\mu >0$ such that for all $x\in \reals^n$, $f(x)-f_* \leq \tfrac{1}{2\mu}\|\nabla f(x)\|^2$. This condition holds in particular for strongly convex function but is much weaker.

\begin{corollary}\label{cor:acd_as_gd}
Under the conditions of Proposition~\ref{prop:acd_as_gd}, if there exists $\mu > 0$ such that for all $x \in \reals^n$, $f(x)-f_*\leq \frac{1}{2\mu}\|\nabla f(x)\|^2$ then after $N \in \N$ iterations
\[f(y_N) -f_* \leq \left(1+\frac{\mu}{L}\right)^{-N}(f(x_0)-f_*).\] 
\end{corollary}

Looking at Proposition~\ref{prop:acd_as_gd} more closely, we notice that when the estimates $\tilde{\mu}_k$ are larger than $\sqrt{L\mu}$, the adaptive accelerated method exhibits an accelerated linear convergence rate $O(1-\sqrt{\tfrac{\mu}{L}})$. It remains to study the convergence of the adaptive method in the regime where $\tilde{\mu}_k$ is small. In this case, we provide another robustness result for the AGM algorithm when the momentum $\beta_k$ is close enough to its classical value~\eqref{eq:costmom}.

\begin{lemma}[Appendix~\ref{proof:acd_34}]\label{lem:acd_34} Let $f \in \mathcal{F}_{\mu,L}$, some iteration number $k\in\N$, and consider Algorithm~\ref{algo:agdp} with 
\[\tfrac{\sqrt{L}-\sqrt[4]{L\mu}}{\sqrt{L}+\sqrt[4]{L\mu}}\leq \beta_k \leq \beta_*=\tfrac{\sqrt{L}-\sqrt{\mu}}{\sqrt{L}+\sqrt{\mu}}.\]
Then, for any $x_k,y_k \in \reals^n$, it holds that
\BEQ
V(x_{k+1},y_{k+1}) \leq \rho V(x_k,y_k)
\EEQ
where $V(x,y) = \frac{L}{2}\|\frac{1}{\sqrt{\rho}}(x-x_*) - \sqrt{\rho}(y-x_*)\|^2 + f(y)-f_* $ and $\rho = \left(1+\left(\frac{\mu}{L}\right)^{\frac{3}{4}}\right)^{-1}$.
\end{lemma}

This lemma guarantees a linear convergence rate $O\left(1-\left(\tfrac{\mu}{L} \right)^{3/4}\right)^k$ that is slower than the accelerated rate with full knowledge of $\mu$ but faster than the gradient rate. We now combine the convergence results for the two regimes of $\tilde{\mu}_k$, and get a global linear convergence rate for~\eqref{eq:Acc2}.

\begin{proposition}[Appendix~\ref{proof:34}]\label{prop:34} Let $f \in \mathcal{F}_{\mu,L}$, and $N\in\N$ be a number of iterations. We consider Algorithm~\ref{algo:agdp} with~\eqref{eq:Acc2}, and let $\{y_k,x_k\}_k$ be the iterates of the method. Then, for any $x_0 \in \reals^n$, such that the sequence $\{\tilde{\mu}_k\}_k$ is well defined, we let $m\in N$ be the first integer such that $\frac{\|\nabla f(y_{m+1})\|^2}{2(f(y_{m+1})-f_*)} \leq \sqrt{L\mu}$, (let $m = \infty$ if this never happens during the $N$ iterations), 
\[ 
f(y_N) - f_* \leq \left\{ \begin{array}{ll}
    \rho_1^{N}\left(\frac{L}{2}\left(\frac{1}{\sqrt{\rho_1}} - \sqrt{\rho_1}\right)^2\|x_0-x_*\|^2  + f(x_0)-f_* \right)\quad & \text{ if } m = 0,\\
     \rho_2^{N}(f(x_0)-f_*) & \text{ if } m = \infty,\\
     C\rho_1^{N-m}\rho_2^{m}(f(x_0) -f_*)& \text{ otherwise,}
\end{array}\right.
\]
where $C =  \left(\left(\tfrac{1}{\rho_1}-1\right)\left(1 + \sqrt{\tfrac{L}{2\mu}} \right)^2 +1\right)$, $\rho_1 = \left(1+\left(\tfrac{\mu}{L}\right)^{\tfrac{3}{4}}\right)^{-1}$ and $\rho_2 = \left(1+\sqrt{\tfrac{\mu}{L}}\right)^{-1}$. \\Otherwise $\nabla f(y_k) = 0$ with $k \in [0,N]$.
\end{proposition}

The previous convergence bound is only valid for \eqref{eq:Acc2} mostly for technical reasons. Indeed the min is present in order to have at most one transition between the regime $\tilde{\mu}_k \geq \sqrt{L\mu}$ and $\tilde{\mu}_k \leq \sqrt{L\mu}$. In practice, however, we didn't observe any difference between the behaviours of~\eqref{eq:Acc1} and that of~\eqref{eq:Acc2}.

\section{Proof mechanisms}\label{sec:proofsmech}


Starting with the work of \cite{Dror14}, computer-aided worst-case analyses of convex optimization methods have provided a generic technique producing convergence rates for many classical first-order algorithms. The results in \citep{Dror14,Tayl17} use an interpolation argument to write the problem of finding the worst case behavior of an algorithm, given a convergence criterion, as a tractable semidefinite program---often referred to as a Performance Estimation Program (PEP). We adapted the technique for generating the complexity bounds on gradient methods with Polyak steps. 

Our proofs were obtained by searching for Lyapunov (or potential) functions (see e.g.~\citep{bansal2019potential} for a recent survey). Due to space constraints, we do not detail how these potentials were obtained here, and refer the reader to the discussions on PEPs in~\citep{pmlr-v99-taylor19a,pmlr-v80-taylor18a} for more details. A related line of works (equivalent in many situations) is that of integral quadratic constraints~\citep{lessard2016analysis}, which leverage results from control theory to perform worst-case complexity analysis. All these approaches were originally developed for non adaptive methods and in what follows, we show how we used the PEP approach for adaptive algorithms. A similar reasoning would allow adapting IQCs for adaptive methods as well.

To fix ideas and illustrate our procedure, we first analyze the worst case complexity of a variant of the classical gradient method with Polyak steps, and show improved convergence bounds compared to classical results (see \cite{Haza19b} for a recent treatment). 
We consider the gradient method with Polyak steps described in Algorithm~\ref{algo:generic_Polyak} with~\eqref{eq:Polyakstep2} for $f \in \mathcal{F}_{\mu,L}$. Notice that there is a factor two in the step-size that is not present in the original Polyak step. This factor simplifies, and improves, the analysis for the convergence in terms of distance to the optimum.


To prove a linear convergence rate, we can focus on the improvement yielded by a single iteration of the form
\begin{equation}\label{eq:poly2iter}
        x_{k+1} := x_k - \gamma_k\nabla f(x_k),
        \quad \mbox{where} \quad
        \gamma_k := 2\frac{f(x_k) - f_*}{\|\nabla f(x_k)\|^2}.
\end{equation}
We seek to bound the worst case (i.e., smallest) decrease in $\|x_{k+1}-x_*\|^2$ relative to $\|x_k-x_*\|^2$ when $x_{k+1}$ is obtained using the iteration in~\eqref{eq:poly2iter} for any function $f \in \mathcal{F}_{\mu,L}$ and any point $x_k$. In other words we seek to solve the following optimization problem
\begin{equation}\label{eq:inf-interp}
\begin{array}{ll}
     \mbox{maximize}& \dfrac{\|x_{k+1}-x_*\|^2}{\|x_k-x_*\|^2}  \\
     \mbox{subject to} & x_{k+1} = x_k - 2\frac{f(x_k)-f_*}{\|\nabla f(x_k)\|^2}\nabla f(x_k),\\
                        & f \in \mathcal{F}_{\mu,L},\; x_k \in \reals^n.
\end{array} 
\end{equation}
in the variables $f \in \mathcal{F}_{\mu,L}$ and $x_k,x_{k+1},x_*,\nabla f(x_k) \in \reals^n$, with parameter $f^*\in\reals$. The following lemma from \citep{Tayl17} shows necessary conditions satisfied by any function $f \in \mathcal{F}_{\mu,L}$.
\begin{lemma}\citep[Theorem 4]{Tayl17}\label{lem:interpol}
Given $f \in \mathcal{F}_{\mu,L}$, for any $(x,y)\in \reals^n \times \reals^n$
\begin{equation*}
\begin{aligned}
 &f(x) - f(y) + \nabla f(x)^T(y-x) + \tfrac{1}{2L}\|\nabla f(x)-\nabla f(y)\|^2\\&+\tfrac{\mu}{2(1-\tfrac{\mu}{L})}\|x-y-\tfrac{1}{L}(\nabla f(x) -\nabla f(y))\|^2 \leq 0
\end{aligned}
\end{equation*} 
\end{lemma}
The key argument in \citep{Dror14,Tayl17} is that the constraint on the regularity of the function $f$ in problem~\eqref{eq:inf-interp} can be replaced by a finite number of inequalities from Lemma~\ref{lem:interpol}. We get an upper bound on the optimum of problem~\eqref{eq:inf-interp} by relaxing the constraint $f \in \mathcal{F}_{\mu,L}$, keeping just two inequalities from Lemma~\ref{lem:interpol} relating $x_k$ and $x_*$ to obtain the following relaxed problem

\begin{equation}\label{eq:pepfinitedim}
\begin{array}{ll}
     \mbox{maximize}& \dfrac{\|x_{k+1}-x_*\|^2}{\|x_k-x_*\|^2}  \\
     \mbox{subject to}
     & f_k - f_* + g_k^T(x_*-x_k) + \tfrac{1}{2L}\|g_k\|^2+\tfrac{\mu}{2(1-\frac{\mu}{L})}\|x_k-x_*-\tfrac{1}{L}g_k\|^2 \leq 0\\
     & f_* - f_k + \tfrac{1}{2L}\|g_k\|^2+\tfrac{\mu}{2(1-\frac{\mu}{L})}\|x_k-x_*-\tfrac{1}{L}g_k\|^2 \leq 0\\
     & x_{k+1} = x_k - 2\frac{f_k-f_*}{\|g_k\|^2}g_k
\end{array} 
\end{equation}
in the variables $x_k,x_*,g_k \in \reals^n$ and $f_k,f_*\in\reals$.
This relaxed problem is finite dimensional, but still depends on the dimension of the ambient space while we are interested in convergence rates independent of the dimension. One of the key insights of the PEP approach is to notice that \eqref{eq:pepfinitedim} can be kernelized, i.e., written in terms of the quadratic variables $X_k=\|x_k-x_*\|^2,\; G_k=\|g_k\|^2,\; GX_k=g_k^T(x_*-x_k)$ in addition to $f_k$ and $f_*$. Indeed, problem~\eqref{eq:pepfinitedim} is equivalent to solving
\begin{equation}\label{eq:pepkernel}
\begin{array}{ll}
     \mbox{maximize}& 1+4\frac{f_k-f_*}{G_k}\frac{GX_k}{X_k} + 4\frac{(f_k-f_*)^2}{G_kX_k}  \\
     \mbox{subject to}& f_k - f_* + GX_k + \tfrac{1}{2L}G_k+\tfrac{\mu}{2(1-\frac{\mu}{L})}\left( X_k+\tfrac{2}{L}GX_k +\tfrac{1}{L^2}G_k \right) \leq 0\\
     & f_* - f_k  + \tfrac{1}{2L}G_k+\tfrac{\mu}{2(1-\frac{\mu}{L})}\left( X_k +\tfrac{2}{L}GX_k +\tfrac{1}{L^2}G_k \right) \leq 0\\
     & \begin{pmatrix}
     X_k & GX_k\\
     GX_k & G_k
     \end{pmatrix} \succcurlyeq 0
\end{array} 
\end{equation}
in the variables $X_k,G_k,GX_k,f_k,f_* \in \reals$.
This new problem has only five real variables but is not readily tractable because of the non-linearity in the objective. By homogeneity we can impose $X_k = 1$ without loss of generality. We introduce a step size variable $\gamma$ to rewrite the problem as
\begin{equation}\label{eq:pepfinal}
\begin{array}{ll}
     \mbox{maximize}&\rho(\gamma) \\
     \mbox{subject to}&\gamma \in \reals
\end{array} 
\end{equation}
where
\begin{equation}\label{eq:rho-gamma}
\begin{array}{rll}
     \rho(\gamma) :=&\mbox{max.} & 1+2\gamma GX_k + 2(f_k-f_*)\gamma  \\
     &\mbox{s.t.} & f_k - f_* + GX_k + \tfrac{1}{2L}G_k+\tfrac{\mu}{2(1-\frac{\mu}{L})}\left( X_k +\tfrac{2}{L}GX_k +\tfrac{1}{L^2}G_k \right) \leq 0\\
     && f_* - f_k + \tfrac{1}{2L}G_k+\tfrac{\mu}{2(1-\frac{\mu}{L})}\left( X_k +\tfrac{2}{L}GX_k +\tfrac{1}{L^2}G_k \right) \leq 0\\
     && \begin{pmatrix}
     X_k & GX_k\\
     GX_k & G_k
     \end{pmatrix} \succcurlyeq 0\\
     && X_k = 1, \; G_k\gamma = 2(f_k-f_*)   
\end{array} 
\end{equation}
which is a semidefinite program. Given $\gamma$, $\rho(\gamma)$ can thus be computed efficiently and our relaxation upper bound on the convergence rate of the method is then given by the maximum value of $\rho(\gamma)$. Note that due to the definition of the step size, we only need to study $\rho(\gamma)$ on the interval $[\tfrac{1}{L},\tfrac{1}{\mu}]$. Figure~\ref{fig:polyak2_rates} (left) plots $\rho(\gamma)$ for fixed values $\mu=0.1$ and $L=1$, and shows (right) the maximum value of $\rho(\gamma)$ for various condition numbers. In this experiment, the worst case convergence rates we obtained numerically appear to perfectly match the bound ${(L-\mu)^2}/{(L+\mu)^2}$.

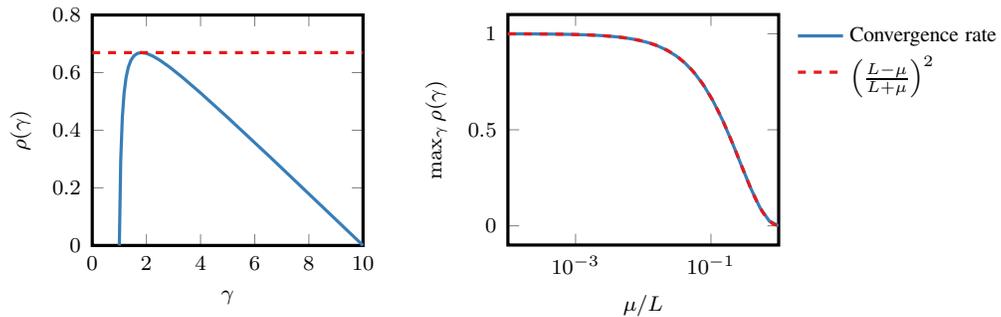
\begin{figure}[ht!]
	\centering
	\begin{tabular}{cc}
		\begin{tikzpicture}
		\begin{axis}[ylabel={$\rho(\gamma)$},xlabel={$\gamma$}, plotOptions,  ymin=0, ymax=0.8,xmin=0,xmax=10,width=.34\linewidth,height=.305\linewidth]
		\addplot [color=colorP1] table [x=gamma,y=rho] {figures/rate_rho_vs_gamma.txt};
		\addplot [color=colorP2, dashed, domain=0:10, samples=2] {(.9)^2/(1.1)^2};
		\end{axis}
		\end{tikzpicture}\\[-4.15cm]
		&
		\begin{tikzpicture}
		\begin{semilogxaxis}[legend pos=outer north east, legend style={draw=none},legend cell align={left}, ylabel={$\max_{\gamma}\rho(\gamma)$},xlabel={$\mu/L$}, plotOptions,  ymin=-.1, ymax=1.1,xmin=1e-4,xmax=1,width=.34\linewidth, height=.305\linewidth]
		\addplot [color=colorP1] table [x=gamma,y=rho] {figures/mu_rate_polyak2.txt}; \addlegendentry{Convergence rate}
		\addplot [color=colorP2, dashed, domain=0.0001:1, samples=50] {(1-x)^2/(1+x)^2}; \addlegendentry{$\left(\tfrac{L-\mu}{L+\mu}\right)^2$}
		\end{semilogxaxis}
		\end{tikzpicture}
		\end{tabular}
		\vspace{-.5cm}\caption{Left: we plot $\rho(\gamma)$, by solving~\eqref{eq:rho-gamma} with $\mu = 0.1$ and $L = 1$. Right: Worst case rate $\max_\gamma\, \rho(\gamma)$, by solving~\eqref{eq:pepfinal}, versus inverse condition number.}
		\label{fig:polyak2_rates}
		\vspace{-.7cm}
\end{figure}

These numerical observations can in fact be proven analytically as follows. Given a target convergence rate $\rho \in [0,1]$, we need to show that
\BEQ\label{proof-ineq}
\|x_{k+1}-x_*\|^2 - \rho \|x_k-x_*\|^2 \leq 0
\EEQ
for all {\em feasible} values of $x_k,x_{k+1},x_*\in\reals^n$, satisfying the constraints of problem~\eqref{eq:pepfinitedim}. In the spirit of the Putinar {\em positivstellensatz} used in sum of squares solutions of semi-algebraic optimization problems \citep{Puti93m,Lass01m,Pari00m}, we seek to write a certificate of the validity of inequality~\eqref{proof-ineq} using a positively weighted sum of valid inequalities satisfied by $x_k,x_{k+1},x_*\in\reals^n$ in~\eqref{eq:pepfinitedim}. Here, this means writing
\begin{equation*}
\begin{aligned}
    &\|x_{k+1}-x_*\|^2 - \rho(\gamma_k)\|x_k-x_*\|^2 = \\
     & \lambda_1\left[ f(x_k) - f_* + \nabla f(x_k)^T(x_*-x_k) + \tfrac{1}{2L}\|\nabla f(x_k)\|^2+ \tfrac{\mu}{2(1-\tfrac{\mu}{L})}\|x_k-x_*-\tfrac{1}{L}\nabla f(x_k)\|^2\right]\\
    +& \lambda_2\left[ f_*-f(x_k) + \tfrac{1}{2L}\|\nabla f(x_k)\|^2+ \tfrac{\mu}{2(1-\tfrac{\mu}{L})}\|x_k-x_*-\tfrac{1}{L}\nabla f(x_k)\|^2\right]\\
    +& \lambda_3\left[2(f(x_k)-f_*) - \gamma_k\|\nabla f(x_k)\|^2 \right] \\
    &\leq  0
    \end{aligned}
\end{equation*} for some $\lambda_1,\lambda_2 \geq 0$, $\lambda_3 \in \reals$, and using the fact $x_{k+1}=x_k-\gamma_k \nabla f(x_k)$ by construction. Through symbolic computations, or by trial and error, inferring a target convergence rate from optimal values of the semidefinite program, the proof consists in showing that we can pick 
\[
\rho(\gamma_k) = \tfrac{(\gamma_k L - 1)(1-\gamma_k\mu)}{\gamma_k(L+\mu)-1},\quad
\lambda_1 = \tfrac{2\gamma_k(\gamma_k L-1)}{\gamma_k(L+\mu) -1},
\quad
\lambda_2 = \tfrac{2\gamma_k(1-\gamma_k\mu)}{\gamma_k(L+\mu) -1} \quad\text{and}\quad 
\lambda_3 =  \tfrac{\gamma_k(2-\gamma_k(L+\mu))}{\gamma_k(L+\mu)-1}.
\]
In practice, the numerical solution of the semidefinite program in~\eqref{eq:rho-gamma} giving $\rho(\gamma)$ can be used to greedily narrow down the list of valid inequalities required by the proof. 

Note that since \eqref{eq:pepfinal} is a semialgebraic problem, we could have used sum-of-squares techniques to prove the convergence rate. However, the multipliers and the rates are fractions in $\gamma_k$. Since one usually doesn't know in advance the form of the denominators, one needs relatively high degree polynomials in the SOS program. This means this approach suffers from the usual SOS issues of poor conditioning and scaling. 
\section{Numerical experiments}
Numerical experiments with our algorithms are provided in Figure~\ref{fig:exps}, respectively on least squares, regularized logistic regression and Lasso problems. For solving the Lasso problems, we used a proximal variant of Algorithm~\ref{algo:agdp}, whose details are provided in Appendix~\ref{proof:prox}. We respectively used the Sonar~\citep{gorman1988analysis} and Musk~\citep{dietterich1997solving} datasets.

In the experiments, when no analytical version of $f_*$ was available (for logistic regression and Lasso), we used ad hoc methods to obtain higher precision estimates of $f_*$. As previously discussed, a fundamental next step is to incorporate successive refinements of a lower bound on $f_*$ (a first step in this direction is for example \citep{Haza19b}). One should notice that vanilla Polyak steps without momentum actually perform very well when they apply (see Appendix~\ref{app:polyak1} for a discussion on the performances of vanilla Polyak steps). We believe that modifying the accelerated Polyak so that it also adapts to the Lipschitz constant could make it more competitive, but the current state of the proofs does not allow it yet.
\def \pltscale {0.275}
\begin{figure}
	\centering
	    \hspace{-0.3cm}\begin{tabular}{lll}
		\begin{tikzpicture}
		\begin{semilogyaxis}[legend pos=outer north east, legend style={draw=none},legend cell align={left}, ylabel={$f-f_*$},xlabel={iterations},xtick={0,1000,2000}, plotOptions,  ymin=1e-11, ymax=1e5,xmin=-10,xmax=2000,width=\pltscale\linewidth, height=\pltscale\linewidth]
		\addplot [color=colorP1] table [x=iter,y=gd] {figures/LSSonar.txt};
		\addplot [color=colorP2] table [x=iter,y=apg] {figures/LSSonar.txt};
		\addplot [color=colorP3] table [x=iter,y=apgmu] {figures/LSSonar.txt};
		\addplot [color=colorP4] table [x=iter,y=pol] {figures/LSSonar.txt};
		\addplot [color=colorP5] table [x=iter,y=adapt] {figures/LSSonar.txt};
		\end{semilogyaxis}
		\end{tikzpicture}
		&
		\hspace{-0.4cm} 
		\begin{tikzpicture}
		\begin{semilogyaxis}[xtick={0,15000,30000}, legend pos=outer north east, legend style={draw=none},legend cell align={left},xlabel={iterations}, plotOptions,  ymin=1e-11, ymax=1e4,xmin=-100,xmax=30000,width=\pltscale\linewidth, height=\pltscale\linewidth]
		\addplot [color=colorP1] table [x=iter,y=gd] {figures/Logit1e-3Sonar.txt};
		\addplot [color=colorP2] table [x=iter,y=apg] {figures/Logit1e-3Sonar.txt};
		\addplot [color=colorP3] table [x=iter,y=apgmu] {figures/Logit1e-3Sonar.txt};
		\addplot [color=colorP4] table [x=iter,y=pol] {figures/Logit1e-3Sonar.txt};
		\addplot [color=colorP5] table [x=iter,y=adapt] {figures/Logit1e-3Sonar.txt};
		\end{semilogyaxis}
		\end{tikzpicture}
		&
		\hspace{-0.4cm} 
		\begin{tikzpicture}
		\begin{semilogyaxis}[legend pos=outer north east, legend style={draw=none},legend cell align={left},xlabel={iterations}, plotOptions,  ymin=1e-11, ymax=1e3,xmin=-10,xmax=2000,xtick={0,1000,2000},width=\pltscale\linewidth, height=\pltscale\linewidth]
		\addplot [color=colorP1] table [x=iter,y=gd] {figures/Lasso1Sonar.txt};
		\addlegendentry{GD}
		\addplot [color=colorP2] table [x=iter,y=apg] {figures/Lasso1Sonar.txt};
		\addlegendentry{AGM-smooth}
		\addplot [color=colorP3] table [x=iter,y=a] {figures/zero.txt};
		\addlegendentry{AGM}
		\addplot [color=colorP5] table [x=iter,y=adapt] {figures/Lasso1Sonar.txt};
		\addlegendentry{Acc Polyak II}
		\addplot [color=colorP4] table [x=iter,y=a] {figures/zero.txt};
		\addlegendentry{Polyak}
		\end{semilogyaxis}
		\end{tikzpicture}
		\\
		\begin{tikzpicture}
		\begin{semilogyaxis}[xtick={0,10000,20000},legend pos=outer north east, legend style={draw=none},legend cell align={left}, ylabel={$f-f_*$},xlabel={iterations}, plotOptions,  ymin=1e-11, ymax=1e5,xmin=-10,xmax=20000,width=\pltscale\linewidth, height=\pltscale\linewidth]
		\addplot [color=colorP1] table [x=iter,y=gd] {figures/LSMusk.txt};
		\addplot [color=colorP2] table [x=iter,y=apg] {figures/LSMusk.txt};
		\addplot [color=colorP3] table [x=iter,y=apgmu] {figures/LSMusk.txt};
		\addplot [color=colorP4] table [x=iter,y=pol] {figures/LSMusk.txt};
		\addplot [color=colorP5] table [x=iter,y=adapt] {figures/LSMusk.txt};
		\end{semilogyaxis}
		\end{tikzpicture}
		&
		\hspace{-0.4cm} 
		\begin{tikzpicture}
		\begin{semilogyaxis}[xtick={0,50000,100000}, legend pos=outer north east, legend style={draw=none},legend cell align={left},xlabel={iterations}, plotOptions,  ymin=1e-11, ymax=1e4,xmin=-1000,xmax=100000,width=\pltscale\linewidth, height=\pltscale\linewidth]
		\addplot [color=colorP1] table [x=iter,y=gd] {figures/Logit1e-3Musk.txt};
		\addplot [color=colorP2] table [x=iter,y=apg] {figures/Logit1e-3Musk.txt};
		\addplot [color=colorP3] table [x=iter,y=apgmu] {figures/Logit1e-3Musk.txt};
		\addplot [color=colorP4] table [x=iter,y=pol] {figures/Logit1e-3Musk.txt};
		\addplot [color=colorP5] table [x=iter,y=adapt] {figures/Logit1e-3Musk.txt};
		\end{semilogyaxis}
		\end{tikzpicture}
		&
		\hspace{-0.4cm} 
		\begin{tikzpicture}
		\begin{semilogyaxis}[
		xtick={0,5000,10000}, xlabel={iterations}, plotOptions,  ymin=1e-11, ymax=1e3,xmin=-100,xmax=10000,width=\pltscale\linewidth, height=\pltscale\linewidth]
		\addplot [color=colorP1] table [x=iter,y=gd] {figures/Lasso1Musk.txt};
		\addplot [color=colorP2] table [x=iter,y=apg] {figures/Lasso1Musk.txt};
		\addplot [color=colorP3] table [x=iter,y=a] {figures/zero.txt};
		\addplot [color=colorP5] table [x=iter,y=adapt] {figures/Lasso1Musk.txt};
		\addplot [color=colorP4] table [x=iter,y=a] {figures/zero.txt};
		\end{semilogyaxis}
		\end{tikzpicture}
		\end{tabular}
		\vspace{-.5cm}
		\caption{Top: Sonar dataset. Bottom: Musk dataset. Left: Least squares. Middle: Logistic regression with Tikhonov regularization (regularization parameter $10^{-3}$). Right: LASSO (regularization parameter $1$). For Polyak steps the best iterate is displayed. No tuning in any of the methods.}
		\label{fig:exps}
		\vspace{-1cm}
\end{figure}
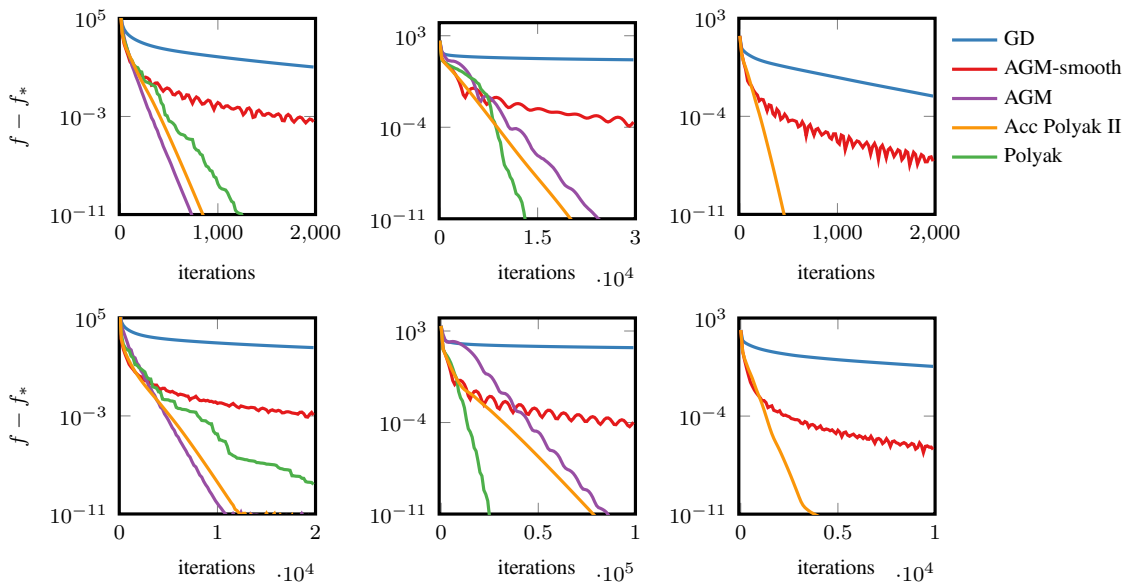 
		
\section{Conclusion and perspectives}
We provided a momentum version of the Polyak steps, with an accelerated linear convergence rate. When $f_*$ is available, this method is easy to implement and requires no tuning at all. On the way, we illustrated the methodology that was used for obtaining those rates, for the special case of a gradient method with Polyak steps. This methodology relies on the recent developments on performance estimation problems~\citep{Dror14,Tayl17}, which we adapted for studying our adaptive methods.

One of the main questions that remains open is to understand whether there exists a way to get the same convergence guarantees without using $f_*$. The robustness result of Lemma~\ref{lem:notbroken} is reassuring in the sense that a misspecified $f_*$ cannot break the algorithm (albeit worsening the convergence rate). We are confident that ideas introduced by \cite{Haza19b} for Polyak steps could be used for our algorithm as well, and could potentially allow dealing with unknown $f_*$ at a reasonable cost. However it still appears as an unnatural trick that adds complexity to the method.

Let us mention that the problem of designing theoretically supported adaptive methods is an open question. We managed to design~\eqref{eq:Polyakstep3}, for which we used our methodology---to find a method that would use Polyak steps to make the primal gap decrease linearly at each iterations---, but designing adaptive accelerated methods appeared as much more daunting task.

Finally, we note that regular Polyak steps do not enjoy a known (working) proximal extension. On the contrary, our results suggest that its accelerated counterparts do work with proximal operators (for minimizing composite objective functions with a non-smooth term). Therefore, developing the theory in this direction is another natural next step.
\clearpage
\paragraph{Codes} The code used to obtain Figures~\ref{fig:polyak2_rates}-\ref{fig:polyak1}-\ref{fig:exps} and to verify proofs is available at \\
{\small \url{https://github.com/mathbarre/PerformanceEstimationPolyakSteps}}. 

\acks{The authors thank Konstantin Mishchenko and Yura Malitsky for insightful discussions on Polyak steps, and comments on a preliminary version of this work. The authors also thank three anonymous reviewers for their constructive feedbacks on the manuscript.

MB acknowledges support from an AMX fellowship. AT acknowledges support from the European Research Council (grant SEQUOIA 724063). AA is at CNRS \& d\'epartement d'informatique, \'Ecole normale sup\'erieure, UMR CNRS 8548, 45 rue d'Ulm 75005 Paris, France,  INRIA  and  PSL  Research  University. AA acknowledges support from the French government under management of Agence Nationale de la Recherche as part of the "Investissements d'avenir" program, reference ANR-19-P3IA-0001 (PRAIRIE 3IA Institute), the ML \& Optimisation joint research initiative with the fonds AXA pour la recherche and Kamet Ventures, as well as a Google focused award.}

\bibliography{MainPerso,mybiblio}
\clearpage
\appendix

\section{Proof of Proposition~\ref{prop:polyak1}}\label{proof:polyak1}
\begin{proof}
For proving the desired result, it is only necessary to consider a single iteration of Algorithm~\ref{algo:generic_Polyak} with~\eqref{eq:Polyakstep2}. We use the following (in)equalities obtained from Lemma~\ref{lem:interpol}:
\begin{itemize}
    \item smoothness and strong convexity between $x_k$ and $x_*$, with multiplier $\lambda_1 = \frac{2\gamma_k(\gamma_k L-1)}{\gamma_k(L+\mu) -1}$:
    \[f(x_k) - f_* + \nabla f(x_k)^T(x_*-x_k) + \tfrac{1}{2L}\|\nabla f(x_k)\|^2+ \tfrac{\mu}{2(1-\tfrac{\mu}{L})}\|x_k-x_*-\tfrac{1}{L}\nabla f(x_k)\|^2 \leq 0,\]
    \item smoothness and strong convexity between $x_*$ and $x_k$, with multiplier $\lambda_2 = \frac{2\gamma_k(1-\gamma_k\mu)}{\gamma_k(L+\mu) -1}$:
    \[f_*-f(x_k) + \tfrac{1}{2L}\|\nabla f(x_k)\|^2+ \tfrac{\mu}{2(1-\tfrac{\mu}{L})}\|x_k-x_*-\tfrac{1}{L}\nabla f(x_k)\|^2 \leq 0,\]
    \item definition of the step-size policy, with multiplier $\lambda_3 = \frac{\gamma_k(2-\gamma_k(L+\mu))}{\gamma_k(L+\mu)-1}$:
    \[ 2(f(x_k)-f_*) - \gamma_k\|\nabla f(x_k)\|^2 = 0.\]
\end{itemize}
Given that $\lambda_1,\lambda_2\geq 0$ (since $\tfrac1L\leq \gamma_k\leq \tfrac1\mu$), the following weighted sum is a valid inequality:
\begin{equation*}
    \begin{aligned}
    0 \geq & \lambda_1\left[ f(x_k) - f_* + \nabla f(x_k)^T(x_*-x_k) + \tfrac{1}{2L}\|\nabla f(x_k)\|^2+ \tfrac{\mu}{2(1-\tfrac{\mu}{L})}\|x_k-x_*-\tfrac{1}{L}\nabla f(x_k)\|^2\right]\\
    &+ \lambda_2\left[ f_*-f(x_k) + \tfrac{1}{2L}\|\nabla f(x_k)\|^2+ \tfrac{\mu}{2(1-\tfrac{\mu}{L})}\|x_k-x_*-\tfrac{1}{L}\nabla f(x_k)\|^2\right]\\
    &+ \lambda_3\left[2(f(x_k)-f_*) - \gamma_k\|\nabla f(x_k)\|^2 \right] .
    \end{aligned}
\end{equation*}

Using the fact that $x_{k+1}=x_k-\gamma_k \nabla f(x_k)$, this weighted sum can be reformulated exactly as
\[ \|x_{k+1}-x_*\|^2 - \rho(\gamma_k)\|x_k-x_*\|^2 \leq 0\]
(one can verify that both expressions are equal) with $\rho(\gamma) = \frac{(\gamma L - 1)(1-\gamma\mu)}{\gamma(L+\mu)-1}$. Therefore, after $N$ iterations, we get $${\|x_N-x_*\|^2 \leq \left(\prod_{i=0}^{N-1}\rho(\gamma_i) \right)\|x_0-x_*\|^2}.$$
In addition, distance to optimality decreases, in the worst-case, with rate $\max_{\gamma}\rho(\gamma)$, with
\[ 
\tfrac{(L-\mu)^2}{(L+\mu)^2}=\max \left\{\rho(\gamma)\,\big|\,\tfrac1L\leq\gamma\leq\tfrac1\mu\right\}.
\]
because $\rho(\gamma)$ is a concave function of $\gamma$ on the interval $[\frac{1}{L},\frac{1}{\mu}]$, as $\rho''(\gamma) = -\frac{2L\mu}{(\gamma(L+\mu) -1)^3}\leq 0$, whose maximum is attained at $\gamma_*=\frac{2}{L+\mu}$. Note that substituting the expression of $\gamma_k$ inside the interpolation inequalities, instead of using it as an independent equality constraints, yields a considerably less tractable result.
\end{proof}

\section{Proof of Proposition~\ref{prop:polyak2}}\label{proof:polyak2}
\begin{proof}Let us consider a single iteration of Algorithm~\ref{algo:generic_Polyak}, with step sizes~\eqref{eq:Polyakstep3}.
The proof is a consequence of the following combination of inequalities obtained from Lemma~\ref{lem:interpol}:
\begin{itemize}
    \item smoothness and strong convexity between $x_k$ and $x_*$, with multiplier $\lambda_1 = \gamma_k \mu(L\gamma_k-1)$:
    \[f(x_k) - f_* + \nabla f(x_k)^T(x_*-x_k) + \tfrac{1}{2L}\|\nabla f(x_k)\|^2+ \tfrac{\mu}{2(1-\tfrac{\mu}{L})}\|x_k-x_*-\tfrac{1}{L}\nabla f(x_k)\|^2 \leq 0,\]
    \item smoothness and strong convexity between $x_{k+1}$ and $x_*$, with multiplier $\lambda_2 = \gamma_k\mu$:
    \begin{equation*}
        \begin{aligned}
         f(x_{k+1}) - f_* + \nabla f(x_{k+1})^T(x_*-x_{k+1}) &+ \tfrac{1}{2L}\|\nabla f(x_{k+1})\|^2\\
         &+ \tfrac{\mu}{2(1-\tfrac{\mu}{L})}\|x_{k+1}-x_*-\tfrac{1}{L}\nabla f(x_{k+1})\|^2 \leq 0,
        \end{aligned}
    \end{equation*}
    \item smoothness and strong convexity between $x_{k+1}$ and $x_k$, with multiplier $\lambda_3 = 1-\gamma_k\mu$:
    \begin{equation*}
        \begin{aligned}
        f(x_{k+1}) - f(x_k) &+ \nabla f(x_{k+1})^T(x_k-x_{k+1}) + \tfrac{1}{2L}\|\nabla f(x_{k+1})-\nabla f(x_k)\|^2\\ &+\tfrac{\mu}{2(1-\tfrac{\mu}{L})}\|x_{k+1}-x_k-\tfrac{1}{L}(\nabla f(x_{k+1}) -\nabla f(x_k))\|^2 \leq 0,
        \end{aligned}
    \end{equation*}
    \item definition of the step-size policy, with multiplier $\lambda_4 = \frac{\gamma_k}{2}((L+\mu)\gamma_k -2)$:
    \[ (2L^2\gamma_k-4L)(f(x_k)-f_*) + \|\nabla f(x_k)\|^2 = 0. \]
\end{itemize}
Given that $\lambda_1,\lambda_2,\lambda_3\geq 0$ (due to $\tfrac{1}{L}\leq\gamma_k\leq \tfrac{2-\tfrac{\mu}{L}}{L}$), the following weighted sum is a valid inequality:
\begin{equation*}
\begin{aligned}
0\geq  \lambda_1 &\left[ f(x_k) - f_* + \nabla f(x_k)^T(x_*-x_k) + \tfrac{1}{2L}\|\nabla f(x_k)\|^2+ \tfrac{\mu}{2(1-\tfrac{\mu}{L})}\|x_k-x_*-\tfrac{1}{L}\nabla f(x_k)\|^2\right]\\
 +\lambda_2 &\bigg[ f(x_{k+1}) - f_* + \nabla f(x_{k+1})^T(x_*-x_{k+1}) + \tfrac{1}{2L}\|\nabla f(x_{k+1})\|^2\\
 &+ \tfrac{\mu}{2(1-\tfrac{\mu}{L})}\|x_{k+1}-x_*-\tfrac{1}{L}\nabla f(x_{k+1})\|^2\bigg]\\
+\lambda_3&\bigg[f(x_{k+1}) - f(x_k) + \nabla f(x_{k+1})^T(x_k-x_{k+1}) + \tfrac{1}{2L}\|\nabla f(x_{k+1})-\nabla f(x_k)\|^2  \\
&+\tfrac{\mu}{2(1-\tfrac{\mu}{L})}\|x_{k+1}-x_k-\tfrac{1}{L}(\nabla f(x_{k+1}) -\nabla f(x_k))\|^2 \bigg]\\
  +\lambda_4 &\left[(2L^2\gamma_k-4L)(f(x_k)-f_*) + \|\nabla f(x_k)\|^2\right].\\
\end{aligned}
\end{equation*}
Using the expression $x_{k+1}=x_k-\gamma_k \nabla f(x_k)$ (without substituting the expression of $\gamma_k$, whose value is encoded through the last equality of the list), this weighted sum can be rewritten exactly as
\begin{equation*}
\begin{aligned}
0\geq & f(x_{k+1})-f_*- \rho(\gamma_k)(f(x_k)-f_*)\\&+\tfrac{1}{2(L-\mu)}\left\|\nabla f(x_{k+1})-L\mu\gamma_k(x_k-x_*) + (\gamma_k(L+\mu) - 1)\nabla f(x_k) \right\|^2
\end{aligned}
\end{equation*}
with $\rho(\gamma) = (L\gamma - 1)\left(L\gamma(3-\gamma(L+\mu))-1\right)$ which, in turns, give
\begin{equation*}
\begin{aligned}
f(x_{k+1})-f_*\leq&  \rho(\gamma_k)(f(x_k)-f_*)\\&-\tfrac{1}{2(L-\mu)}\left\|\nabla f(x_{k+1})-L\mu\gamma_k(x_k-x_*) + (\gamma_k(L+\mu) - 1)\nabla f(x_k) \right\|^2\\
\leq&  \rho(\gamma_k)(f(x_k)-f_*).
\end{aligned}
\end{equation*}
Therefore, after $N$ iterations, we get $${f(x_N)-f_* \leq \left(\prod_{i=0}^{N-1}\rho(\gamma_i) \right)(f(x_0)-f_*).}$$
Finally, the worst-case convergence rate is $\max_{\gamma}\rho(\gamma)$ on the interval $[\tfrac{1}{L},\tfrac{2-{\mu}/{L}}{L}]$, for which
\[ {\tfrac{(L-\mu)^2}{(L+\mu)^2}=\max\left\{\rho(\gamma)\,\big|\, \tfrac1L\leq\gamma\leq\tfrac{2-{\mu}/{L}}{L} \right\}.}\]
The proof follows from the following steps:
\begin{itemize}
    \item First, on the boundaries of the interval: (i) $\rho(\frac{1}{L}) = 0$ and (ii) $\rho(\frac{2-\frac{\mu}{L}}{L}) = \frac{\left(L-\mu\right)^4}{L^4} \leq \frac{\left(L - \mu\right)^2}{\left(L + \mu\right)^2}$.
    \item Secondly, in the interior of the interval: $\rho'(\gamma) = L(3L\gamma-2)(2-(L+\mu)\gamma)$ is zero at $\gamma_* = \frac{2}{L+\mu}$ (inside the interval).
    \item Therefore ${\rho(\gamma_*) = \frac{\left(L-\mu\right)^2}{\left(L+\mu\right)^2}}$ and this is the maximum on the interval.
\end{itemize}
\end{proof}
\section{Proof of \S~\ref{sec:acc}}
\subsection{Proof of Lemma~\ref{lem:notbroken}}\label{proof:notbroken}
\begin{proof}
In this section, we use $\rho=1-\mu/L$. The proof consists in combining the following inequalities obtained from Lemma~\ref{lem:interpol}:
\begin{itemize}
    \item smoothness and strong convexity between $x_k$ and $y_k$ with multiplier $\lambda_1 = \rho$:
    \begin{equation*}
        \begin{aligned}
        f(x_k) - f(y_k) &+ \nabla f(x_k)^T(y_k-x_k) + \tfrac{1}{2L}\|\nabla f(x_k)-\nabla f(y_k)\|^2\\&+ \tfrac{\mu}{2(1-\tfrac{\mu}{L})}\|x_k-y_k-\tfrac{1}{L}(\nabla f(x_k)-\nabla f(y_k) )\|^2 \leq 0,
        \end{aligned}
    \end{equation*} 
    \item smoothness and strong convexity between $y_{k+1}$ and $x_*$ with multiplier $\lambda_2 = 1-\rho$:
    \begin{equation*}
        \begin{aligned}
        f(y_{k+1}) - f_* &+ \nabla f(y_{k+1})^T(x_*-y_{k+1}) + \tfrac{1}{2L}\|\nabla f(y_{k+1})\|^2\\
        &+ \tfrac{\mu}{2(1-\tfrac{\mu}{L})}\|y_{k+1}-x_*-\tfrac{1}{L}\nabla f(y_{k+1})\|^2 \leq 0,
        \end{aligned}
    \end{equation*} 
    
    \item smoothness and strong convexity between $y_{k+1}$ and $x_k$ with multiplier $\lambda_3 = \rho$:
    \begin{equation*}
        \begin{aligned}
        f(y_{k+1}) - f(x_k) &+ \nabla f(y_{k+1})^T(x_k-y_{k+1}) + \tfrac{1}{2L}\|\nabla f(y_{k+1})-\nabla f(x_k)\|^2\\&+\tfrac{\mu}{2(1-\tfrac{\mu}{L})}\|y_{k+1}-x_k-\tfrac{1}{L}(\nabla f(y_{k+1}) -\nabla f(x_k))\|^2 \leq 0.
        \end{aligned}
    \end{equation*}
\end{itemize}
Given that $\lambda_1,\lambda_2,\lambda_3\geq0$, the following weighted sum is a valid inequality
\begin{equation*}
    \begin{aligned}
    0\geq &\lambda_1 \bigg[ f(x_k) - f(y_k) + \nabla f(x_k)^T(y_k-x_k) + \tfrac{1}{2L}\|\nabla f(x_k)-\nabla f(y_k)\|^2 \\
    &+ \tfrac{\mu}{2(1-\frac{\mu}{L})}\|x_k-y_k-\tfrac{1}{L}(\nabla f(x_k)-\nabla f(y_k) )\|^2\bigg]\\
    & + \lambda_2\bigg[f(y_{k+1}) - f_* + \nabla f(y_{k+1})^T(x_*-y_{k+1}) + \tfrac{1}{2L}\|\nabla f(y_{k+1})\|^2 \\
    &+\tfrac{\mu}{2(1-\tfrac{\mu}{L})}\|y_{k+1}-x_*-\tfrac{1}{L}\nabla f(y_{k+1})\|^2\bigg]\\
    &+\lambda_3 \bigg[ f(y_{k+1}) - f(x_k) + \nabla f(y_{k+1})^T(x_k-y_{k+1}) + \tfrac{1}{2L}\|\nabla f(y_{k+1})-\nabla f(x_k)\|^2\\
    &  +\tfrac{\mu}{2(1-\tfrac{\mu}{L})}\|y_{k+1}-x_k-\tfrac{1}{L}(\nabla f(y_{k+1}) -\nabla f(x_k))\|^2\bigg],
    \end{aligned}
\end{equation*}
which can be reformulated exactly, using the notation \BEAS
V(x,y)&=&f(y)-f_*+\tfrac{L-\mu}{2}\|x-y\|^2\\
y_{k+1} &=& x_{k} - \tfrac{1}{L}\nabla f(x_k)\\
x_{k+1} &=& y_{k+1} + \beta_k(y_{k+1}-y_k)\\
\EEAS
along with the expression of $\rho$, in the form
\begin{equation*}
    \begin{aligned}
    0\geq &V(x_{k+1},y_{k+1}) - \rho V(x_k,y_{k}) \\&+\tfrac1{2(L-\mu)}\left\|(1-\rho ) (\nabla f(x_{k})-L(x_{k}-x_*))+ \nabla f(y_{k+1})\right\|^2\\&+\tfrac{\rho}{2 (L-\mu )}\left\| \nabla f(y_{k}) - \nabla f(x_{k}) + \mu(x_{k} - y_{k})\right\|^2\\
    &+ \tfrac{(1-\beta^2)\rho}{2L}\|\nabla f(x_{k}) + L(y_{k}-x_{k}) \|^2.\\
    \end{aligned}
\end{equation*}
Therefore, using the assumption $\beta_k \in [0,1]$, we finally arrive to the desired
\[ V(x_{k+1},y_{k+1}) \leq \rho V(x_{k},y_{k}).\]
\end{proof}
\subsection{Proof of Lemma~\ref{lem:acc_classic}}\label{proof:acc_classic}
\begin{proof}
In this setting, we write $\rho(x) = \frac{1}{1+\frac{x}{L}}$. The proof consists in the following combination of inequalities obtained from Lemma~\ref{lem:interpol}:
\begin{itemize}
    \item smoothness and convexity between $y_{k+1}$ and $x_{k}$ with multiplier $\lambda_1 = \rho(\tilde{\mu}_k)$:
    \[f(y_{k+1}) - f(x_{k}) + \nabla f(y_{k+1})^T(x_{k}-y_{k+1}) + \tfrac{1}{2L}\|\nabla f(x_{k})-\nabla f(y_{k+1})\|^2 \leq 0,\]
    \item convexity between $x_{k}$ and $y_{k}$ with multiplier $\lambda_2 = \rho(\tilde{\mu}_k)$:
    \[f(x_{k}) - f(y_{k}) + \nabla f(x_{k})^T(y_{k}-x_{k}) \leq 0,\]
    \item definition of $\tilde{\mu}_k$ with multiplier $\lambda_3 = \frac{1-\rho(\tilde{\mu}_k)}{2\tilde{\mu}_k}$:
    \[2\tilde{\mu}_k(f(y_{k+1})-f_*) - \|\nabla f(y_{k+1})\|^2 \leq 0\] (we use an inequality so that it also holds for $\tilde{\mu}_{k}=\min\{  \tilde{\mu}_{k-1},\tfrac{\|\nabla f(y_{k+1})\|^2}{2(f(y_{k+1})-f_*)}\}$).
\end{itemize}
The weighted sum is a valid inequality given that $\lambda_1,\lambda_2,\lambda_3\geq 0$:
\begin{equation*}
\begin{aligned}
0\geq  \lambda_1& \left[ f(y_{k+1}) - f(x_k) + \nabla f(y_{k+1})^T(x_k-y_{k+1}) + \tfrac{1}{2L}\|\nabla f(x_k)-\nabla f(y_{k+1})\|^2\right]\\
&+ \lambda_2 \left[f(x_k) - f(y_k) + \nabla f(x_k)^T(y_k-x_k)\right]\\
&+\lambda_3\left[2\tilde{\mu}_k(f(y_{k+1})-f_*) - \|\nabla f(y_{k+1})\|^2 \right],
\end{aligned}
\end{equation*}
which can be reformulated exactly, using the notation \BEAS
V(x,y)&=&f(y)-f_*+\tfrac{L}{2}\|x-y\|^2\\
y_{k+1} &=& x_{k} - \tfrac{1}{L}\nabla f(x_k)\\
x_{k+1} &=& y_{k+1} + \beta_k(y_{k+1}-y_k)\\
\beta_k &=& \frac{\sqrt{L} - \sqrt{\tilde{\mu}_k} }{\sqrt{L} + \sqrt{\tilde{\mu}_k}}
\EEAS
along with the expression for $\rho(x)$, in the form
\begin{equation*}
\begin{aligned}
0\geq & V(x_{k+1},y_{k+1}) - \rho(\tilde{\mu}_k)V(x_k,y_k) \\
&+\tfrac{\left(4 L^2 \sqrt{\tfrac{\tilde{\mu}_k}{L}}-L \left(\tilde{\mu}_k-2 \tilde{\mu}_k \sqrt{\frac{\tilde{\mu}_k}{L}}\right)-\tilde{\mu}_k^2\right)}{2 L^2 (L+\tilde{\mu}_k) \left(\sqrt{\frac{\tilde{\mu}_k}{L}}+1\right)^2}\|\nabla f(x_k)+L (y_k-x_k)\|^2,
\end{aligned}
\end{equation*}
which, in turns, is equivalent to
\begin{equation*}
\begin{aligned}
V(x_{k+1},y_{k+1}) \leq & \rho(\tilde{\mu}_k)V(x_k,y_k) -\tfrac{\left(4 L^2 \sqrt{\frac{\tilde{\mu}_k}{L}}-L \left(\tilde{\mu}_k-2 \tilde{\mu}_k \sqrt{\frac{\tilde{\mu}_k}{L}}\right)-\tilde{\mu}_k^2\right)}{2 L^2 (L+\tilde{\mu}_k) \left(\sqrt{\frac{\tilde{\mu}_k}{L}}+1\right)^2}\|\nabla f(x_k)+L (y_k-x_k)\|^2,\\
\leq & \rho(\tilde{\mu}_k)V(x_k,y_k)
\end{aligned}
\end{equation*}
where the inequality follows from the sign of the term we removed, so it remains to show that \[4 L^2 \sqrt{\tfrac{\tilde{\mu}_k}{L}}-L \left(\tilde{\mu}_k-2 \tilde{\mu}_k \sqrt{\tfrac{\tilde{\mu}_k}{L}}\right)-\tilde{\mu}_k^2 \geq 0 \quad \forall\tilde{\mu}_k \in [0,L].\] Indeed, evaluating the sign of the previous expression boils down to study that of $g(x) = 4 \sqrt{x}-\left(x-2 x \sqrt{x}\right)-x^2$ on $[0,1]$, which follows from:
\[ g(x) \geq 3\sqrt{x}-x\sqrt{x} \geq 0 \quad \forall x\in [0,1].\]
\end{proof}
\subsection{Proof of Lemma~\ref{lem:acd_34}}\label{proof:acd_34}
\begin{proof}Our statement follows from a weighted sum of inequalities obtained from Lemma~\ref{lem:interpol}:
\begin{itemize}
    \item smoothness and strong convexity between $y_{k+1}$ and $x_{k}$, with multiplier $\lambda_1 = 1$:
    \begin{equation*} 
        \begin{aligned}
        f(y_{k+1}) - f(x_{k}) &+ \nabla f(y_{k+1})^T(x_{k}-y_{k+1}) + \tfrac{1}{2L}\|\nabla f(y_{k+1}) - \nabla f(x_{k})\|^2 \\&+ \tfrac{\mu}{2(1-\frac{\mu}{L})}\|y_{k+1}-x_{k}-\tfrac{1}{L}(\nabla f(y_{k+1})-\nabla f(x_{k}))\|^2 \leq 0,
        \end{aligned}
    \end{equation*}
    \item smoothness and strong convexity between $x_{k}$ and $x_*$, with multiplier $\lambda_2 = 1-\rho$:
    \begin{equation*}
        \begin{aligned}
        f(x_{k}) - f_* &+ \nabla f(x_{k})^T(x_*-x_{k})+\tfrac1{2L}\|\nabla f(x_{k})\|^2\\&+\tfrac{\mu}{2(1-\tfrac{\mu}{L})}\|x_{k}-x_*-\tfrac1L \nabla f(x_{k})\|^2\leq 0,
        \end{aligned}
    \end{equation*}
    \item convexity between $x_{k}$ and $y_{k}$, with multiplier $\lambda_3 = \rho$: 
    \[f(x_{k}) - f(y_{k})+ \nabla f(x_{k})^T(y_{k}-x_{k})\leq 0.\]
\end{itemize}
The weighted sum is a valid inequality given that $\lambda_1,\lambda_2,\lambda_3\geq0$:
\begin{equation*}
\begin{aligned}
0\geq  \lambda_1 &\bigg[ f(y_{k+1}) - f(x_{k}) + \nabla f(y_{k+1})^T(x_{k}-y_{k+1}) + \tfrac{1}{2L}\|\nabla f(y_{k+1}) - \nabla f(x_{k})\|^2\\
& +\tfrac{\mu}{2(1-\tfrac{\mu}{L})}\|y_{k+1}-x_{k}-\frac{1}{L}(\nabla f(y_{k+1})-\nabla f(x_{k}))\|^2 \bigg]\\
+ \lambda_2 &\bigg[f(x_{k}) - f_*+ \nabla f(x_{k})^T(x_*-x_{k})+\tfrac1{2L}\|\nabla f(x_{k})\|^2 \\ & +\tfrac{\mu}{2(1-\tfrac{\mu}{L})}\|x_{k}-x_*-\tfrac1L \nabla f(x_{k})\|^2\bigg]\\
+\lambda_3&\left[f(x_{k}) - f(y_{k})+ \nabla f(x_{k})^T(y_{k}-x_{k}) \right].
\end{aligned}
\end{equation*}
This inequality can be reformulated using the notations  
\BEAS
V(x,y)&=&f(y)-f_*+\tfrac{L}{2}\|\tfrac{1}{\sqrt{\rho}}(x-x_*)-\sqrt{\rho}(y-x_*) \|^2\\
y_{k+1} &=& x_{k} - \tfrac{1}{L}\nabla f(x_k)\\
x_{k+1} &=& y_{k+1} + \beta_k(y_{k+1}-y_k)\\
\beta &=& \beta_k 
\EEAS
in the form
\begin{equation*}
\begin{aligned}
0\geq & V(x_{k+1},y_{k+1}) -\rho V(x_{k},y_{k}) +\tfrac{1}{2(L-\mu)}\|\nabla f(y_{k+1})\|^2+\tfrac{1-\rho}{2L}\|\nabla f(x_{k})\|^2 \\
 & +\tfrac{L \left(\rho ^3-\beta ^2\right)}{2 \rho }\| (y_{k}-x_*) + \tfrac{\beta  \rho -\beta  (\beta +1)+\rho ^2}{\beta ^2-\rho ^3} (x_{k}-x_*) + \tfrac{\beta ^2-\beta  \rho +\beta -\rho ^2}{\beta ^2 L-L \rho ^3} \nabla f(x_{k})\|^2\\
&+\tfrac{L^2 (1-\rho) \left(\tfrac{\mu}{L}  \rho  (2 \beta  \rho -\beta  (\beta +2)+\rho )+ (\rho -1) (\beta -\rho )^2\right)}{2 \left(\rho ^3-\beta ^2\right) (L-\mu )} \|x_{k}-x_*-\tfrac1L\nabla f(x_{k})\|^2.
\end{aligned}
\end{equation*}
It is then direct to reach
\begin{equation*}
\begin{aligned}
 V(x_{k+1},y_{k+1}) \leq&  \rho V(x_{k},y_{k}) -\tfrac{1}{2(L-\mu)}\|\nabla f(y_{k+1})\|^2-\tfrac{1-\rho}{2L}\|\nabla f(x_{k})\|^2 \\
 & -\tfrac{L \left(\rho ^3-\beta ^2\right)}{2 \rho }\| (y_{k}-x_*) + \tfrac{\beta  \rho -\beta  (\beta +1)+\rho ^2}{\beta ^2-\rho ^3} (x_{k}-x_*) + \tfrac{\beta ^2-\beta  \rho +\beta -\rho ^2}{\beta ^2 L-L \rho ^3} \nabla f(x_{k})\|^2\\
&-\tfrac{L^2 (1-\rho) \left(\tfrac{\mu}{L}  \rho  (2 \beta  \rho -\beta  (\beta +2)+\rho )+ (\rho -1) (\beta -\rho )^2\right)}{2 \left(\rho ^3-\beta ^2\right) (L-\mu )} \|x_{k}-x_*-\tfrac1L\nabla f(x_{k})\|^2\\
  \leq &\rho V(x_{k},y_k),
\end{aligned}
\end{equation*}
where we used the facts that the following coefficients were nonnegative (proofs below) on the domain of interest:
\begin{itemize}
    \item $\tfrac{1}{2(L-\mu)}\geq 0$ (clear from the assumption $\mu\leq L$),
    \item $ \tfrac{1-\rho}{L}\geq 0$ (clear from $\rho\leq 1$),
    \item $\tfrac{L \left(\rho ^3-\beta ^2\right)}{2 \rho }\geq 0$ follows from $\left(\rho ^3-\beta ^2\right)\geq 0$, proved below,
    \item $\tfrac{L^2 (1-\rho) \left(\tfrac{\mu}{L}  \rho  (2 \beta  \rho -\beta  (\beta +2)+\rho )+ (\rho -1) (\beta -\rho )^2\right)}{2 \left(\rho ^3-\beta ^2\right) (L-\mu )}\geq0$ follows from previous points along with \[\tfrac{\mu}{L}  \rho  (2 \beta  \rho -\beta  (\beta +2)+\rho )+ (\rho -1) (\beta -\rho )^2\geq 0,\] which is alo proved below.
\end{itemize}
The missing proofs are as follow. First, let us define $\kappa:=\tfrac\mu L\in[0,1]$, the (inverse) condition number, and recall that we want to prove the expressions above to be nonnegative when $\rho=\tfrac{1}{1+\kappa^{3/4}}$ and $\beta_-\leq \beta\leq \beta_+$ with $\beta_- = \tfrac{\sqrt{1} - \sqrt[4]{\kappa}}{\sqrt{1} + \sqrt[4]{\kappa}}$ and $\beta_+ = \tfrac{\sqrt{1} - \sqrt{\kappa}}{\sqrt{1} + \sqrt{\kappa}}$.
\begin{itemize}
    \item To show that $\rho ^3-\beta ^2\geq 0$, let us remark that the expression is a second order polynomial in the variable $\beta$ with negative curvature. Therefore, its minimum values are achieved on the boundary of the interval, and it is sufficient to show $\rho ^3-\beta_- ^2\geq 0$ and $\rho ^3-\beta_+ ^2\geq 0$ for establishing our claim. For the case $\beta=\beta_-$, we get:
    \begin{equation*}
    \begin{aligned}
    \rho ^3-\beta_- ^2 &= \tfrac{\kappa ^{1/4} \left(4-8 \kappa ^{1/4}+9 \sqrt{\kappa }-4 \kappa ^{3/4}-4 \kappa +9 \kappa ^{5/4}-8 \kappa ^{3/2}+4 \kappa ^{7/4}-\kappa ^2\right)}{\left(1+\kappa ^{1/4}\right)^3 \left(1-\kappa ^{1/4}+\sqrt{\kappa }\right)^3},
    \end{aligned}
    \end{equation*}
    and we need to show that $\left(4-8 \kappa ^{1/4}+9 \sqrt{\kappa }-4 \kappa ^{3/4}-4 \kappa +9 \kappa ^{5/4}-8 \kappa ^{3/2}+4 \kappa ^{7/4}-\kappa ^2\right)$ is non negative for all $\kappa\in[0,1]$. For showing that, we perform the change of variable $x\leftarrow \kappa^{1/4}$ (which is invertible since $\kappa\in[0,1]$), and study the polynomial 
    \[ 
    p_1(x)= -x^8+4x^7-8x^6+9x^5-4x^4-4x^3+9x^2-8x+4,
    \]
    such that
\BEAS
p_1(x)&\geq& 3x^7 -8x^6+9x^5-4x^4-4x^3+9x^2-8x+4\\
        & =& 3x^7 -8x^6+9x^5-4x^4-4x^3+5x^2 +4(1-x)^2\\
        & \geq& 3x^7 -8x^6+9x^5-4x^4-4x^3+5x^2\\
        & \geq& 3x^7 -8x^6+9x^5-4x^4+x^3\\
        & = & 3x^7 -8x^6+5x^5 +x^3(2x-1)^2\\
        & \geq& x^5(3x^2 -8x+5)\\
        & = &x^5(1-x)(5-3x)\\
        & \geq& 0,
\EEAS
hence finally $\rho ^3-\beta_- ^2\geq0$. For the case $\beta=\beta_+$, we obtain:
    \begin{equation*}
    \begin{aligned}
    \rho ^3-\beta_+ ^2 &=\tfrac{\sqrt{\kappa } \left(4-3 \kappa ^{1/4}+6 \kappa ^{3/4}-3 \kappa -3 \kappa ^{5/4}+6 \kappa ^{3/2}-\kappa ^{7/4}-3 \kappa ^2+2 \kappa ^{9/4}-\kappa ^{11/4}\right)}{\left(1+\kappa ^{1/4}\right)^3 \left(1+\sqrt{\kappa }\right)^2 \left(1-\kappa ^{1/4}+\sqrt{\kappa }\right)^3},
    \end{aligned}
    \end{equation*}
    and we need to show that
    \[\left(4-3 \kappa ^{1/4}+6 \kappa ^{3/4}-3 \kappa -3 \kappa ^{5/4}+6 \kappa ^{3/2}-\kappa ^{7/4}-3 \kappa ^2+2 \kappa ^{9/4}-\kappa ^{11/4}\right)\]
    is nonnegative for all $\kappa\in[0,1]$. After changing variable $x\leftarrow \kappa^{1/4}$ (which is invertible since $\kappa\in[0,1]$), we study the polynomial
\[p_2(x)= -x^{11}+2x^9-3x^8-x^7+6x^6-3x^5-3x^4+6x^3-3x+4\]
such that
\BEAS    
p_2(x)&\geq& x^9-3x^8-x^7+6x^6-3x^5-3x^4+6x^3-3x+4\\
    &\geq& x^9-3x^8-x^7+6x^6-3x^5-3x^4+6x^3+1\\
    &\geq& x^9-3x^8-x^7+6x^6+1\\
    &\geq& x^9+2x^6+1\\
    &\geq& 0,
\EEAS
hence $\rho ^3-\beta_+ ^2\geq0$.
    \item Similarly, the expression $p_3(\kappa)=\left(\kappa  \rho  (2 \beta  \rho -\beta  (\beta +2)+\rho )+(\rho -1) (\beta -\rho )^2\right)$ is also a second order polynomial in $\beta$, with leading coefficient
    \[ -(1-\rho)-\kappa\rho\leq -(1-\rho)\leq 0. \]
    Therefore, this quadratic function is also concave and we only need to verify the inequality on the boundary of the interval $[\beta_-,\beta_+]$. In the case $\beta=\beta_-$, we get:
    \begin{equation*}
    \begin{aligned}
    p_3(\beta_-) &= \tfrac{\left(1-\sqrt{\kappa }+\kappa ^{3/4}\right) \kappa ^{7/4}}{\left(1+\kappa ^{1/4}\right)^3 \left(1-\kappa ^{1/4}+\sqrt{\kappa }\right)^3}\geq 0.
    \end{aligned}
    \end{equation*}
    For case $\beta=\beta_+$, we obtain:
    \begin{equation*}
    \begin{aligned}
    p_3(\beta_+)&=\tfrac{\kappa ^{3/2} \left(4-7 \kappa ^{1/4}+4 \sqrt{\kappa }+5 \kappa ^{3/4}-7 \kappa +3 \kappa ^{5/4}+2 \kappa ^{3/2}-\kappa ^{7/4}+\kappa ^2\right)}{\left(1+\kappa ^{1/4}\right)^3 \left(1+\sqrt{\kappa }\right)^2 \left(1-\kappa ^{1/4}+\sqrt{\kappa }\right)^3},
    \end{aligned}
    \end{equation*}
    and we need to show that $\left(\kappa ^2-\kappa ^{7/4}+2 \kappa ^{3/2}+3 \kappa ^{5/4}-7 \kappa +5 \kappa ^{3/4}+4 \sqrt{\kappa }-7 \sqrt[4]{\kappa }+4\right)$ is nonnegative for $\kappa \in [0,1]$. We change  variables $x\leftarrow \kappa^{1/4}$ (which is invertible since $\kappa\in[0,1]$), and study the polynomial \[p_4(x) = x^8-x^{7}+2 x ^{6}+3x ^{5}-7 x^4 +5 x ^{3}+4 x^2 -7 x+4\] on the interval $[0,1]$:
\BEAS    
p_4(x) &=& x^8 -x^7 +2x^6 +3x^5 -7x^4 + 5x^3  +x +4(1-x)^2\\
    &\geq & x^3(x^5 -x^4 +2x^3 +3x^2 -7x + {5})\\
    &= & x^3(x^5 -x^4 +2x^3 -x^2 +x + {1} + 4(1-x)^2)\\
    &\geq & x^3(x^5 + x^3 + {1} + 4(1-x)^2)\\
    &\geq & 0,
\EEAS
hence $p_3(\beta_+)\geq 0$, which concludes the proof.
\end{itemize}
\end{proof}

\subsection{Proof of Proposition~\ref{prop:34}}\label{proof:34}
\begin{proof}
The case $m = 0$ results from Lemma~\ref{lem:acd_34} applied recursively and the case $m = \infty$ result from Proposition~\ref{prop:acd_as_gd}. In the following we consider that $m \in [1,N]$. Then for $(y_k,x_k)_{k\in[m+1,N]}$, 
\[
\frac{\sqrt{L}-\sqrt[4]{L\mu}}{\sqrt{L}+\sqrt[4]{L\mu}}\leq \beta_{k-1} \leq \frac{\sqrt{L}-\sqrt{\mu}}{\sqrt{L}+\sqrt{\mu}}
\]
and Lemma~\ref{lem:acd_34} applies so 
\[ 
f(y_N) - f_* \leq \rho_1^{N-m}\left(\frac{L}{2}\|\frac{1}{\sqrt{\rho_1}}(x_m-x_*) - \sqrt{\rho_1}(y_m-x_*)\|^2 + f(y_m)-f_* \right) 
\]
and we have
\BEAS    
    &&\frac{L}{2}\|\frac{1}{\sqrt{\rho_1}}(x_m-x_*) - \sqrt{\rho_1}(y_m-x_*)\|^2 + f(y_m)-f_*  \\ 
    &&= \frac{L}{2}\left(\frac{1}{\rho_1}-1\right)\|x_m-x_*\|^2 - \frac{L}{2}(1-\rho_1)\|y_m-x_*\|^2 +\frac{L}{2}\|x_m-y_m\|^2 + f(y_m)-f_* \\
    && \leq\frac{L}{2}\left(\frac{1}{\rho_1}-1\right)\|x_m-x_*\|^2+\frac{L}{2}\|x_m-y_m\|^2 + f(y_m)-f_*  \\
    && \leq \frac{L}{2}\left(\frac{1}{\rho_1}-1\right)\left(\|x_m-y_m\| + \|y_m-x_*\| \right)^2
    +\frac{L}{2}\|x_m-y_m\|^2 + f(y_m)-f_*  \\
    && \leq \left(\frac{1}{\rho_1}-1\right)\left(\sqrt{\frac{L}{2}}\|x_m-y_m\| + \sqrt{\frac{L}{2\mu}}\sqrt{f(y_m)-f_*} \right)^2 +\frac{L}{2}\|x_m-y_m\|^2 + f(y_m)-f_*
\EEAS
We can now apply Corollary~\ref{cor:acd_as_gd}. From the definition of $m$, we have \[ 2(f(y_k)-f_*) \leq \frac{1}{\sqrt{L\mu}}\|\nabla f(y_k)\|^2 \text{ for all }k\in[1,m].\] 
Therefore, by denoting $\rho_2 =\left(1+\sqrt{\frac{\mu}{L}} \right)^{-1}$, we have the following inequalities
\BEAS
\frac{L}{2}\|x_m-y_m\|^2 + f(y_m)-f_* &\leq & \rho_2^m(f(x_0)-f_*),\\
\sqrt{\frac{L}{2}}\|x_m-y_m\| & \leq & \rho_2^{m/2}\sqrt{f(x_0)-f_*},\\
\sqrt{\frac{L}{2\mu}}\sqrt{f(y_m)-f_*} & \leq & \sqrt{\frac{L}{2\mu}}\rho_2^{m/2}\sqrt{f(x_0)-f_*},
\EEAS
which leads to
\BEAS
&&\frac{L}{2}\|\frac{1}{\sqrt{\rho_1}}(x_m-x_*) - \sqrt{\rho_1}(y_m-x_*)\|^2 + f(y_m)-f_* \\ 
&& \leq \left(\left(\frac{1}{\rho_1}-1\right)\left(1 + \sqrt{\frac{L}{2\mu}} \right)^2 +1\right)\rho_2^m(f(x_0) -f_*),
\EEAS
reaching the desired result.
\end{proof}
\subsection{Proximal variants}\label{proof:prox}
A natural extension of smooth and strongly convex optimization is the case composite optimization
\[ \min_{x\in\reals^n} \{ F(x)\equiv f(x)+h(x)\},\]
where $f\in\mathcal{F}_{\mu,L}$ and $h\in\mathcal{F}_{0,\infty}$ is a proper convex function with proximal operator available.

\begin{algorithm}[!ht]
\caption{Proximal accelerated gradient method}
\label{algo:prox_agdp}
\begin{algorithmic}
\STATE \algorithmicrequire\;$x_0 \in \reals^n$, $f_*\in \reals$, $L$ smoothness constant.
\STATE $y_0 = x_0$, 
\FOR{$k \geq 0$}
\STATE $y_{k+1} = \mathrm{prox}_{h/L} \left(x_k - \frac{1}{L}\nabla f(x_k)\right)$
\STATE compute $\tilde{\mu}_k$ and $\beta_k=\tfrac{\sqrt{L}-\sqrt{\tilde{\mu}_k}}{\sqrt{L}+\sqrt{\tilde{\mu}_k}}$
\STATE $x_{k+1} = y_{k+1} + \beta_k(y_{k+1}-y_k)$
\ENDFOR
\STATE \algorithmicensure\; $y_{k+1}$
\end{algorithmic}
\end{algorithm}

We used the proximal version of AGM with constant momentum. It is of the same form as Algorithm~\ref{algo:agdp} but the gradient step is combined with a proximal step. We extended our estimate $\tilde{\mu}_k$ the following way. Given $F = f + h$ where $f \in \mathcal{F}_{\mu,L}$ and $h$ a proper convex function that is proximable, $\tilde{\mu}_k = \frac{\mathcal{D}(y_{k+1},L)}{2(F(y_{k+1}) - F_*)}$ where $\mathcal{D}(x,L) = -2 L\;\underset{y}{\min}\; \left[ \langle \nabla f(x),y-x \rangle + \tfrac{L}{2}\|x-y\|^2 +h(y) -h(x) \right] $. Notice that when $h = 0$ the previous formula is exactly \eqref{eq:Acc1}. Also, when they are well defined these estimates still belong to $[\mu,L]$ \citep{Kari16}.

\subsection{Study of standard Polyak steps}\label{app:polyak1}

From numerical experiments, we noticed that~\eqref{eq:Polyakstep2} was actually typically performing only slightly better than vanilla gradient descent. From a worst-case point of view, this is expected. However, our experiments (see Figure~\ref{fig:intro-adap}-\ref{fig:exps}) suggest that regular Polyak steps \eqref{eq:Polyakstep1} actually perform much better than one could expect from its worst-case guarantees.

In this section, we provide a tentative explanation of this behavior, through experiments on a toy example. Figure~\ref{fig:polyak1} (top) was obtained by running the methods on a least squares problem (we used a rescaled version of the Sonar dataset, with regularity parameters $L=1$ and $\mu = 0.01$).

Similar in spirit as in Figure~\ref{fig:polyak2_rates} (left), we provide, in Figure~\ref{fig:polyak1}, the worst-case ratio of $\|x_{k+1}-x_*\|^2/\|x_k-x_*\|^2$ (by solving~\eqref{eq:inf-interp} numerically for regular Polyak steps). One can observe that the worst case rate (using distances to optimum as the criterion) is slightly worse than that of~\eqref{eq:Polyakstep2} (note that this rate can be improved through the use of refined Lyapunov functions). 

In Figure~\ref{fig:polyak1}, we provide the distributions of step size magnitudes observed through the optimization process on the toy example. One can notice that the distribution does not fully concentrate around the worst-case value (the value of $\gamma$ that achieves the worst-case) for~\eqref{eq:Polyakstep1}. A large proportion of effective step size values are even located in regions of fast convergence. On the contrary, for~\eqref{eq:Polyakstep2}, the distribution is much more concentrated around its worst-case. Those distributions strongly suggest that worst case analyses might not be the best way to explain the good practical behaviors of such adaptive methods. 

\begin{figure}
    \centering
    \hspace{-1.1cm}\begin{tikzpicture}
		\begin{semilogyaxis}[legend pos=outer north east, legend style={draw=none},legend cell align={left},xlabel={iterations},ylabel={$f-f_*$}, plotOptions,  ymin=1e-11, ymax=1e2,xmin=-1,xmax=400,xtick={0,200,400},width=0.48\linewidth, height=0.41\linewidth]
		\addplot [color=colorP1] table [x=iter,y=gd] {figures/LSSonareasy1e-2.txt};
		\addlegendentry{GD}
		\addplot [color=colorP2] table [x=iter,y=apg] {figures/LSSonareasy1e-2.txt};
		\addlegendentry{AGM-smooth}
		\addplot [color=colorP4] table [x=iter,y=pol] {figures/LSSonareasy1e-2.txt};
		\addlegendentry{Polyak}
		\addplot [color=colorP6] table [x=iter,y=pol2] {figures/LSSonareasy1e-2.txt};
		\addlegendentry{Variant I}
		\end{semilogyaxis}
		\end{tikzpicture}\\
		\hspace{1.86cm}
		\begin{tikzpicture}
		\begin{axis}[legend pos=outer north east, legend style={draw=none},legend cell align={left},ylabel={},xlabel={$\gamma$},every axis plot/.append style={ultra thick}, plotOptions,  ymin=0, ymax=1,xmin=0,xmax=50,width=.48\linewidth,height=.41\linewidth]
		\addplot[ybar,bar width=3pt,fill=black,opacity=0.5] table [x=bincenter, y=prop] {figures/bins_polyak1.txt};
    \addlegendentry{Distribution of observed $\{\gamma_k\}_k$}
		\addplot [color=colorP1] table [x=gamma,y=rho] {figures/gamma_rho_polyak1_1e-2.txt};
		\addlegendentry{Convergence rate $\rho(\gamma)$}
		\addplot [color=colorP2, dashed, domain=0:50, samples=2] {0.9901/(1.01)^2};
		\addlegendentry{$\max_\gamma \,\rho(\gamma)=\frac{L^2-L\mu+\mu^2}{(L+\mu)^2}$}
		\end{axis}
		\end{tikzpicture}   \\
		\hspace{1.9cm}\begin{tikzpicture}
		\begin{axis}[legend pos=outer north east, legend style={draw=none},legend cell align={left},ylabel={},xlabel={$\frac{\gamma}{2}$},every axis plot/.append style={ultra thick}, plotOptions,  ymin=0, ymax=1,xmin=0,xmax=50,width=.48\linewidth,height=.41\linewidth]
		\addplot[ybar,bar width=3pt,fill=black,opacity=0.5] table [x=bincenter, y=prop] {figures/bins_polyak2.txt};
    \addlegendentry{Distribution of observed $\{\gamma_k\}_k$}
		\addplot [color=colorP1] table [x=gamma,y=rho] {figures/gamma_rho_polyak2_1e-2.txt};
		\addlegendentry{Convergence rate $\rho(\gamma)$}
		\addplot [color=colorP2, dashed, domain=0:50, samples=2] {0.99^2/(1.01)^2};
		\addlegendentry{$\max_\gamma\,\rho(\gamma)=\frac{(L-\mu)^2}{(L+\mu)^2}$}
		
		\end{axis}
		\end{tikzpicture}
    
    \caption{Top: Least squares on rescaled Sonar dataset ($L=1$ and $\mu=0.01$). Middle: $\rho(\gamma)$ for~\eqref{eq:Polyakstep1} (blue)---computed numerically following the methodology of \S~\ref{sec:proofsmech} with fixed $L=1$ and $\mu=0.01$. Distribution of effective step size magnitudes (black) used throughout the 150 iterations of \eqref{eq:Polyakstep1} appearing in (top).
    Bottom: $\rho(\gamma)$ for~\eqref{eq:Polyakstep2} (blue)---with $L=1$ and $\mu = 0.01$. Distribution of effective step size magnitudes (black) used throughout the 400 iterations of~\eqref{eq:Polyakstep2} appearing in (top). }
    \label{fig:polyak1}
\end{figure}
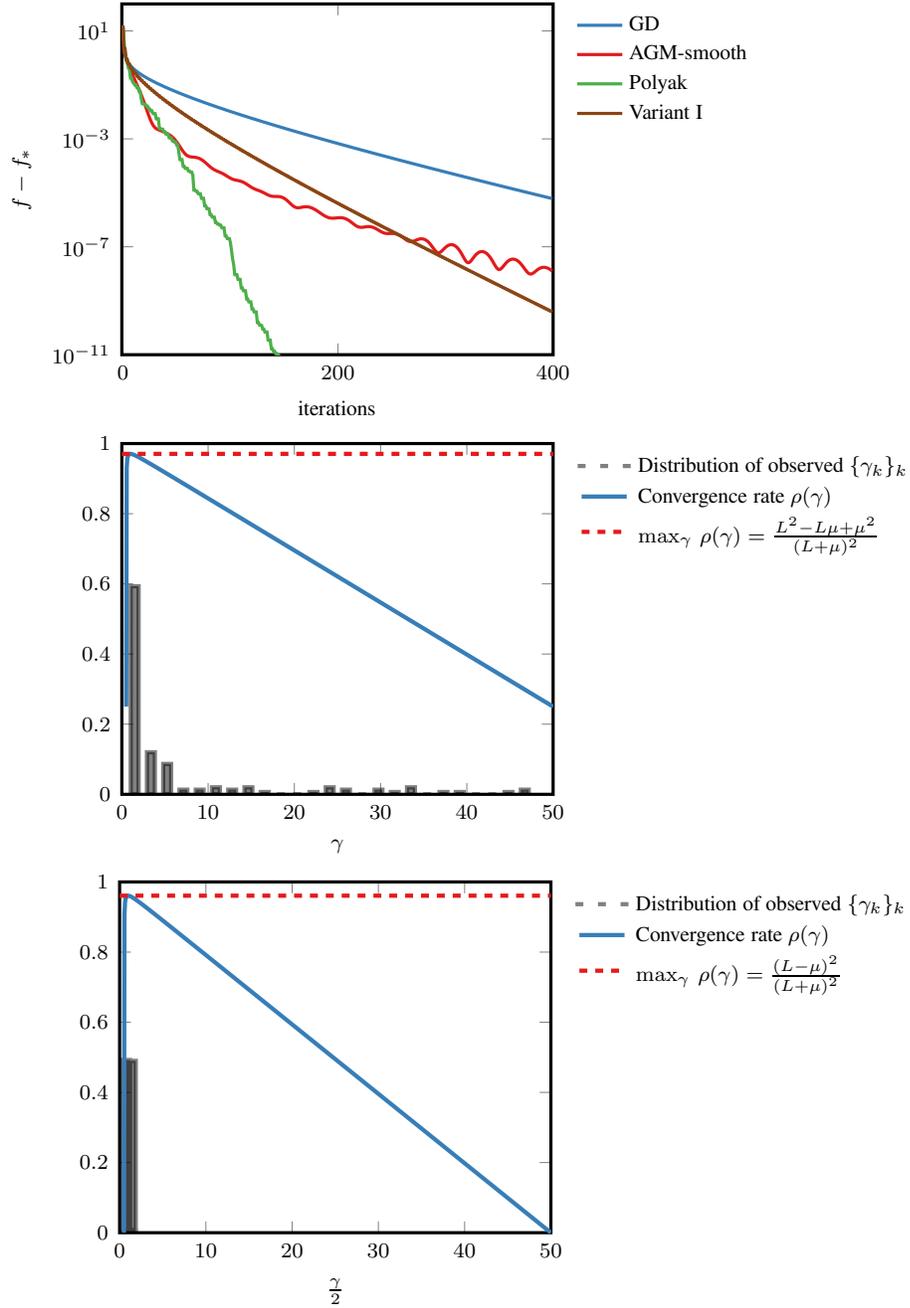
\end{document}

%% file: defs.tex
\definecolor{ddarkbrown}{rgb}{0.5,0.2,0.05} \definecolor{bbluegray}{rgb}{0.05,0,0.5}

\renewenvironment{proof}{\textbf{Proof.}}{\QED\bigskip}

\newcommand{\BEAS}{\begin{eqnarray*}}
\newcommand{\EEAS}{\end{eqnarray*}}
\newcommand{\BEA}{\begin{eqnarray}}
\newcommand{\EEA}{\end{eqnarray}}
\newcommand{\BEQ}{\begin{equation}}
\newcommand{\EEQ}{\end{equation}}
\newcommand{\BIT}{\begin{itemize}}
\newcommand{\EIT}{\end{itemize}}
\newcommand{\BNUM}{\begin{enumerate}}
\newcommand{\ENUM}{\end{enumerate}}

\newcommand{\BA}{\begin{array}}
\newcommand{\EA}{\end{array}}

\newcommand{\refp}[1]{(\ref{#1})}

\newcommand{\cf}{{\it cf.}}
\newcommand{\eg}{{\it e.g.}}
\newcommand{\ie}{{\it i.e.}}
\newcommand{\etc}{{\it etc.}}
\newcommand{\ones}{\mathbf 1}

\newcommand{\reals}{{\mathbb R}}
\newcommand{\sreals}{\scriptsize{\mbox{\bf R}}}
\newcommand{\integers}{{\mbox{\bf Z}}}
\newcommand{\eqbydef}{\mathrel{\stackrel{\Delta}{=}}}
\newcommand{\complex}{{\mbox{\bf C}}}
\newcommand{\symm}{{\mbox{\bf S}}}  

\newcommand{\Span}{\mbox{\textrm{span}}}
\newcommand{\Range}{\mbox{\textrm{range}}}
\newcommand{\nullspace}{{\mathcal N}}
\newcommand{\range}{{\mathcal R}}
\newcommand{\diam}{\mathop{\bf radius}}
\newcommand{\sphere}{{\mathbb S}}
\newcommand{\Nullspace}{\mbox{\textrm{nullspace}}}
\newcommand{\Rank}{\mathop{\bf Rank}}
\newcommand{\NumRank}{\mathop{\bf NumRank}}
\newcommand{\NumCard}{\mathop{\bf NumCard}}
\newcommand{\Card}{\mathop{\bf Card}}
\newcommand{\Tr}{\mathop{\bf Tr}}
\newcommand{\diag}{\mathop{\bf diag}}
\newcommand{\lambdamax}{{\lambda_{\rm max}}}
\newcommand{\lambdamin}{\lambda_{\rm min}}
\newcommand{\idm}{\mathbf{I}}

\newcommand{\Expect}{\textstyle{\bf E}}
\newcommand{\Median}{\textstyle\mathop{\bf M}}
\newcommand{\Prob}{\mathop{\bf Prob}}
\newcommand{\erf}{\mathop{\bf erf}}

\newcommand{\Co}{{\mathop {\bf Co}}}
\newcommand{\co}{{\mathop {\bf Co}}}
\newcommand{\Var}{\mathop{\bf var{}}}
\newcommand{\dist}{\mathop{\bf dist{}}}
\newcommand{\Ltwo}{{\bf L}_2}
\newcommand{\QED}{~~\rule[-1pt]{6pt}{6pt}}\def\qed{\QED}
\newcommand{\approxleq}{\mathrel{\smash{\makebox[0pt][l]{\raisebox{-3.4pt}{\small$\sim$}}}{\raisebox{1.1pt}{$<$}}}}
\newcommand{\argmin}{\mathop{\rm argmin}}
\newcommand{\epi}{\mathop{\bf epi}}
\newcommand{\var}{\mathop{\bf var}}

\newcommand{\vol}{\mathop{\bf vol}}
\newcommand{\Vol}{\mathop{\bf vol}}

\newcommand{\dom}{\mathop{\bf dom}}
\newcommand{\aff}{\mathop{\bf aff}}
\newcommand{\cl}{\mathop{\bf cl}}
\newcommand{\Angle}{\mathop{\bf angle}}
\newcommand{\intr}{\mathop{\bf int}}
\newcommand{\relint}{\mathop{\bf rel int}}
\newcommand{\bd}{\mathop{\bf bd}}
\newcommand{\vect}{\mathop{\bf vec}}
\newcommand{\dsp}{\displaystyle}
\newcommand{\foequal}{\simeq}
\newcommand{\VOL}{{\mbox{\bf vol}}}
\newcommand{\argmax}{\mathop{\rm argmax}}
\newcommand{\xopt}{x^{\rm opt}}

\newcommand{\Xb}{{\mbox{\bf X}}}
\newcommand{\xst}{x^\star}
\newcommand{\varphist}{\varphi^\star}
\newcommand{\lambdast}{\lambda^\star}
\newcommand{\Zst}{Z^\star}
\newcommand{\fstar}{f^\star}
\newcommand{\xstar}{x^\star}
\newcommand{\xc}{x^\star}
\newcommand{\lambdac}{\lambda^\star}
\newcommand{\lambdaopt}{\lambda^{\rm opt}}

\newcommand{\geqK}{\mathrel{\succeq_K}}
\newcommand{\gK}{\mathrel{\succ_K}}
\newcommand{\leqK}{\mathrel{\preceq_K}}
\newcommand{\lK}{\mathrel{\prec_K}}
\newcommand{\geqKst}{\mathrel{\succeq_{K^*}}}
\newcommand{\gKst}{\mathrel{\succ_{K^*}}}
\newcommand{\leqKst}{\mathrel{\preceq_{K^*}}}
\newcommand{\lKst}{\mathrel{\prec_{K^*}}}
\newcommand{\geqL}{\mathrel{\succeq_L}}
\newcommand{\gL}{\mathrel{\succ_L}}
\newcommand{\leqL}{\mathrel{\preceq_L}}
\newcommand{\lL}{\mathrel{\prec_L}}
\newcommand{\geqLst}{\mathrel{\succeq_{L^*}}}
\newcommand{\gLst}{\mathrel{\succ_{L^*}}}
\newcommand{\leqLst}{\mathrel{\preceq_{L^*}}}
\newcommand{\lLst}{\mathrel{\prec_{L^*}}}

\newcommand{\realsp}{\mathbf{R}_+^n}
\newcommand{\intrealsp}{\int_{\mathbf{R}_+^n}}

%% file: related_work.tex
\subsection{Related works}

\paragraph*{Gradient and accelerated gradient methods.}
For smooth optimization problems, simple line search strategies provide accelerated algorithms that adapt to the local gradient Lipschitz constant~\citep{Nest13b} and explicit adaptive complexity bounds can be derived for certain variants using the mean root Lipschitz constant \citep{scheinberg2014fast}.

\paragraph*{Restarts.}
For smooth and strongly convex optimization problems (or more generally problems satisfying H\"olderian error bounds), accelerated methods with optimal complexity bounds require knowledge of the strong convexity constant to compute iterates \citep{Nest13b,Nest18}. In particular, \cite{Arje16} show that this information is necessary when using oblivious steps. This quantity can be hard to estimate and a lot of effort has been put in the development of adaptive optimization methods preserving fast convergence rates \citep{Lin14,Ferc16,Roul17}. All these works are based on restart strategies \citep{ODon15,Nest13b} and although they exhibit fast theoretical convergence rates, they often contain parameters that have to be tuned in order to get good practical results, or require additional information on the function itself (e.g., its minimum $f^*$). Once again, while on paper the complexity of restart schemes is nearly optimal, the presence of an outer loop generally limits their capacity to adapt to the function's local regularity and significantly affects empirical performance.

\paragraph*{Quasi-Newton methods.} 

An important family of adaptive algorithms is composed with quasi-Newton methods. As the name suggests, these methods try to mimic the behavior of Newton schemes, by constructing an estimate of the hessian at the current point, using previous gradients. The most notable quasi-Newton method is certainly L-BFGS \citep{Liu89}. These commonly used algorithms exhibit very fast empirical converge rates but only classical convergence rates comparable to that of gradient descent have been proven at this point \citep{Byrd87}.

\paragraph*{Conjugate gradient methods.} Conjugate gradient methods are probably among the most famous examples of adaptive algorithm. Firstly introduced for quadratic minimization~\citep{Hest52}, and motivated by nice theoretical guarantees (such as finite-time convergence), many variants have been introduced for going beyond quadratics~\citep{Flet64,Poly69b,Flet87}---see, for example, the nice survey~\citep{Hag06}. Roughly speaking, at each iteration, the method constructs an update direction based on the gradient at the current iterate, and on the knowledge of the previous search directions. The next iterate is obtained by line-search in the update direction. Whereas conjugate gradient methods are widely used in practice (e.g \cite{Rodi01,Volk04,Zhao15}), and perform very well when they applies, there are barely any non-asymptotic convergence guarantees for those methods beyond unconstrained quadratic minimization.

\paragraph*{Polyak step-sizes.}
When the optimal value of the objective function value is known, a well-known adaptive strategy consists in using the so-called ``Polyak step-sizes''---see e.g.,~\citep[Section 5.3.2]{polyak1987introduction} or~\citep{nedic2001incremental,Boyd03b}. The method consists in iterating gradient steps with step-sizes proportional to the primal gap at the current iterate. As opposed to most adaptive gradient methods mentioned above, this method comes with explicit theoretical properties, even beyond the quadratic optimization case.

\paragraph*{Barzilai-Borwein step-sizes.} The Barzilai-Borwein~\citep{barzilai1988two,Flet05} method consists in gradient steps with adaptive step-sizes. It is another case with complete theory for quadratic optimization, but barely any performance guarantees in non-quadratic cases (it is even known to diverge on some problem instances). 
\paragraph*{Adaptive gradient steps}
In \citep{Mali19} the authors developed a step-size policy that adapts to the local geometry, together with nice theoretical guarantees.